\documentclass[12pt]{amsart}
\usepackage{amssymb}
\usepackage{amsmath}
\usepackage{amsthm}

\makeatletter
\def\LaTeX{\leavevmode L\raise.42ex
    \hbox{\kern-.3em\size{\sf@size}{0pt}\selectfont A}\kern-.15em\TeX}
\makeatother

\newcommand{\BibTeX}{{\rm B\kern-.05em{\sc
          i\kern-.025emb}\kern-.08em\TeX}}

\makeatletter
\def\@currentlabel{2.1}\label{e:dispaa}
\def\@currentlabel{2.21}\label{e:dispau}
\def\@currentlabel{2.22}\label{e:dispav}
\def\@currentlabel{2.23}\label{e:dispaw}
\def\@currentlabel{2.24}\label{e:dispax}
\def\theequation{\thesection.\@arabic\c@equation}
\makeatother

\renewcommand{\theequation}{\arabic{section}.\arabic{equation}}

\newtheorem{thm}{Theorem}[section]
\newtheorem{lem}[thm]{Lemma}

\newtheorem{prop}[thm]{Proposition}
\theoremstyle{definition}

\newtheorem{rem}[thm]{Remark}

\newcommand{\R}{\mathbb{R}}

\newcommand{\p}{\partial}

\begin{document}
\title[Radial symmetry]
{Radial symmetry of positive entire solutions of a fourth order elliptic equation with
a singular nonlinearity}
\author{Zongming Guo}
\address{Department of Mathematics, Henan Normal University, Xinxiang, 453007, China}
\email{gzm@htu.cn}
\author{Long Wei}
\address{Department of Mathematics, Hangzhou Dianzi University,
 Xiasha, 310018, China}
\email{alongwei@163.com}
\author{Feng Zhou}
\address{Center for PDEs and Department of Mathematics, East China Normal University, Shanghai, 200241, China}
\email{fzhou@math.ecnu.edu.cn}

\thanks{$^*$ Research of the first author is supported by NSFC
(11171092, 11571093).
Research of the second author is supported by Zhejiang Provincial Natural Science Foundation of China under Grants No.
LY14A010029 and NSFC (11101111).
Research of the third author is supported by NSFC (11271133, 11431005) 
 and Shanghai Key Laboratory of PMMP}
 
\date{\today}
\def\baselinestretch{1}
\begin{abstract}
The necessary and sufficient conditions for a regular positive entire solution $u$ of the
biharmonic equation:
\begin{equation}
\label{0.1}
-\Delta^2 u=u^{-p} \;\; \mbox{in $\R^N \; (N \geq 3)$}, \;\; p>1
\end{equation}
to be a radially symmetric solution are obtained via the moving plane method (MPM) of a system of equations. It is well-known that for any $a>0$, \eqref{0.1} admits a unique
minimal positive entire radial solution ${\underline u}_a (r)$ and a family of non-minimal positive entire radial solutions $u_a (r)$ such that $u_a (0)={\underline u}_a (0)=a$ and
$u_a (r) \geq {\underline u}_a (r)$ for $r \in (0, \infty)$. Moreover, the asymptotic behaviors of ${\underline u}_a (r)$ and $u_a (r)$ at $r=\infty$ are also known.
We will see in this paper that the asymptotic behaviors similar to those of ${\underline u}_a (r)$ and $u_a (r)$ at $r=\infty$ can determine the radial symmetry of a general regular positive entire solution $u$ of \eqref{0.1}.
The precisely asymptotic behaviors of $u (x)$ and $-\Delta u (x)$ at $|x|=\infty$ need to be established
such that the moving-plane procedure can be started. We provide the necessary and sufficient conditions not only for a regular positive entire solution $u$ of \eqref{0.1} to be the minimal entire radial solution,
but also for $u$ to be a non-minimal entire radial solution.
\end{abstract}
\maketitle
\baselineskip 18pt

\section{Introduction}
\setcounter{equation}{0}

We consider radial symmetry of positive entire solutions of the equation
\begin{equation}
\label{1.1-0}
-\Delta^2 u=u^{-p} \;\;\; \mbox{in $\R^N$},
\end{equation}
where $N=3$, $1<p<3$ and $N \geq 4$, $p>1$. The necessary and sufficient conditions
for a positive entire solution of \eqref{1.1-0} to be a positive entire radially symmetric solution are established.

Equation \eqref{1.1-0} has been extensively studied in recent years, see, for example, \cite{Ch, CEGM, CX,DFG,Gue,GW1,GW2,LY, MW, MR, Mo} and the references therein. It arises in the study of the deflection of charged
plates in electrostatic actuators in the modeling of electrostatic
micro-electromechanical systems (MEMS) (see \cite{LY1, Pe} and the references therein).
It is known from \cite{DFG} that for $N=3$ and $1<p<3$;
$N \geq 4$ and $p>1$  \eqref{1.1-0} admits a singular
entire radial solution:
$$U_s (r)=L r^{\alpha}, \;\;\; r=|x|,$$
where and in the following,
\begin{equation}
\label{alpha-L}
\alpha=\frac{4}{p+1}, \;\;\;
L=\big[\alpha (2-\alpha)(N-2+\alpha)(N-4+\alpha)\big]^{-\frac{1}{p+1}}.
\end{equation}
Moreover,
for any $a>0$, there is a unique ${\tilde b}:=b(a)>0$ such that the problem
\begin{equation}
\label{1.2}
\left \{ \begin{array}{ll} -\Delta^2 u=u^{-p} \;\; \mbox{in $\R^N$}, \\
u(0)=a, \; u'(0)=0, \; \Delta u(0)=b, \; u''' (0)=0
\end{array} \right.
\end{equation}
has a unique positive radial solution $u_{a,{\tilde b}} (r)$ satisfying
\begin{equation}
\label{condition}
\lim_{r \to \infty} r^{-\alpha} u_{a, {\tilde b}} (r)=L.
\end{equation}
It is also known from \cite{DFG} that for any $b<{\tilde b}$, \eqref{1.2} does not admit an entire radial solution; for any
$b>{\tilde b}$, \eqref{1.2} admits a unique entire radial solution $u_{a, b} (r)$ which has
the growth rate $O(r^2)$ at $r=\infty$. Therefore we see that
the behaviors of the minimal and non-minimal entire solutions at $\infty$ are different. A comparison
principle (Lemma 3.2 in \cite{MR}) ensures that $u_{a, b}>u_{a, {\tilde b}}$ in $(0, \infty)$ for $b>{\tilde b}$. These imply that for any $a>0$, $u_{a,{\tilde b}}$ is the (unique) minimal
positive entire radial solution of \eqref{1.1-0} and $\{u_{a, b}\}_{b>{\tilde b}}$ are a family of entire non-minimal radial solutions of \eqref{1.1-0}. Meanwhile, the comparison principle also implies that for any $b_1>b_2>{\tilde b}$,
$u_{a, b_1}>u_{a, b_2}$ in $(0, \infty)$.
The stability of positive entire solutions of \eqref{1.1-0} has also been studied in \cite{GW2} and the references therein.

In this paper, we are interested in the relationship between the radial symmetry and the asymptotic behavior at $\infty$ of a positive entire solution of \eqref{1.1-0}. We will see that
if a positive regular entire solution $u$
of \eqref{1.1-0} admits the asymptotic behavior as that of the minimal entire radial solution of \eqref{1.1-0}, it is actually the minimal entire radial solution of \eqref{1.1-0} with respect to some $x_* \in \R^N$.
Meanwhile, if a positive regular entire solution $u$
of \eqref{1.1-0} admits the asymptotic behavior as that of a non-minimal entire radial solution of \eqref{1.1-0}, it is actually a non-minimal entire radial solution of \eqref{1.1-0} with respect to some $x_* \in \R^N$.

Our main results are the following theorems.

\begin{thm}
\label{main}
Let $u \in C^4 (\R^N)$ be a positive entire solution of \eqref{1.1-0} and
\begin{equation} \label{pN} p \in \left\{ \begin{array}{lll}
(1,\frac{N+2}{6-N}], &&\mbox{for $N=3$ or $5$,}\vspace{1mm}\\
(1,3] \cup (7,\infty),&\quad& \mbox{for $N=4$,}\vspace{1mm}\\
(1,\infty), &&\mbox{for $N\ge 6$.}\end{array}
\right.
\end{equation}
Then $u$ is the minimal radial entire solution of
\eqref{1.1-0} with the initial value $u(x_*)$ at some $x_* \in \R^N$ (i.e. $u(x)=u(r)$
with $r=|x-x_*|$) if and only if
\begin{equation}
\label{1.4}
\lim_{|x| \to \infty} \Big[ |x|^{-\alpha} u(x)-L \Big]=0.
\end{equation}
\end{thm}

Our results for $p=7$ and $N=4$; $p \in (\frac{N+2}{6-N},p^*)$ and $N=3$, $4$ or $5$ are a little different, where we denote
\begin{equation}\label{p*}p^*:=\frac{N+3}{5-N}=\left\{ \begin{array}{lll}
3,&\quad& \mbox{for $N=3$,}\\
7, &&\mbox{for $N=4$,}\\
\infty, &&\mbox{for $N=5$.}\end{array}
\right.
\end{equation}

\begin{thm}
\label{main-1}
Let $p=7$ and $N =4$; $u \in C^4 (\R^4)$ be a positive entire solution of \eqref{1.1-0}. Then $u$ is the minimal radial entire solution of \eqref{1.1-0} with the initial value
$u(x_*)$ at some $x_* \in \R^4$ if and only if there exists $0<\epsilon_0<\frac{1}{10}$ such that
\begin{equation}
\label{1.6}
|x|^{-\alpha} u(x)-L=o \Big(|x|^{-\epsilon_0} \Big) \;\; \mbox{as $|x| \to \infty$}.
\end{equation}
\end{thm}

\begin{thm}
\label{main-2}
Let $p \in (\frac{N+2}{6-N}, p^*)$ and $N=3$, $4$ or $5$; $u \in C^4 (\R^N)$ be a positive entire solution of \eqref{1.1-0}. Then $u$ is the minimal radial entire solution of \eqref{1.1-0} with the initial value
$u(x_*)$ at some $x_* \in \R^N$ if and only if
\begin{equation}
\label{1.6-1}
|x|^{-\alpha} u(x)-L=o \Big(|x|^{5-N-2\alpha} \Big) \;\; \mbox{as $|x| \to \infty$}.
\end{equation}
\end{thm}
Note that $5-N-2\alpha\in (-1,0)$ when $N =3$, $4$ or $5$ and $p \in (\frac{N+2}{6-N}, p^*)$. \medskip

The following theorem provides the necessary and sufficient conditions for a positive entire solution of \eqref{1.1-0} to be a non-minimal positive radial entire solution of \eqref{1.1-0}.

\begin{thm}
\label{main-3}
Let $u \in C^4 (\R^N)$ be a positive entire solution of \eqref{1.1-0} with $N=3$ and $1<p<3$; $N \geq 4$ and $p>1$. Then $u$ is an entire radial solution about some $x_* \in \R^N$, but is not the minimal positive entire radial solution about $x_*$ of \eqref{1.1-0}, if and only if there exists $D>0$ such that
\begin{equation}
\label{condition-nonminimal}
\lim_{|x| \to \infty} \Big[ |x|^{-2} u(x)-D \Big]=0.
\end{equation}
The constant $D$ then determines a particular non-minimal positive entire radial solution.
\end{thm}

Theorems \ref{main}-\ref{main-3} show that the asymptotic behavior given in \eqref{1.4}, \eqref{1.6}, \eqref{1.6-1} or \eqref{condition-nonminimal} near $\infty$ of a positive entire solution $u$ of \eqref{1.1-0} determines its radial symmetry with respect to some $x_* \in \R^N$, which seem to be the first such kinds of results for problem \eqref{1.1-0}.

Let us comment on some related results. The semilinear equations
$$-\Delta u=u^p \;\; \mbox{in $\R^N \; (N \geq 3)$}, \;\; p>\frac{N+2}{N-2} \leqno(P)$$
and
$$ \Delta u=u^{-p} \;\; \mbox{in $\R^N \; (N \geq 2)$}, \;\; p>0 \leqno(Q)$$
have been studied in the past few decades.
Some sufficient conditions for a regular positive entire solution of (P) and (Q)
to be an entire radial solution are given in \cite{Zou} for (P) provided $p \in (\frac{N+2}{N-2}, \frac{N+1}{N-3})$
and in \cite{GW3} for (Q) provided $p>0$ respectively.
The results in \cite{Zou} were generalized to $p \geq \frac{N}{N-4}$ for $N \geq 5$ in \cite{Guo}.
Recently, the necessary and sufficient conditions for an entire solution
$u$ of the equation
$$\Delta^2 u=8(N-2)(N-4) e^u \;\;\mbox{in $\R^N \;(N \geq 5)$} \leqno(P_1)$$
to be the entire radial solution of $(P_1)$ with the initial value at some $x_* \in \R^N$ are provided in \cite{GHZ}.
Note that \eqref{1.1-0} can be written to the following system of equations:
\begin{equation}
\label{1.7}
\left \{\begin{array}{ll} -\Delta u=w \;\;\; &\mbox{in $\R^N$}, \\
-\Delta w=-u^{-p} \;\;\; &\mbox{in $\R^N$}.
\end{array}
\right.
\end{equation}
As in  \cite{GHZ}, we use the moving plane method for a system of equations to obtain our
results, but we need to do more delicate estimates for the solution $u$ and $\Delta u$ near $\infty$, since (Q) has a more
complicated structure of solutions than (P1). We discuss not only the minimal solution but also the non-minimal solutions in this paper. Such estimates we need to do are more complicated since they rely on two parameters $p$ and $N$.
Moreover, for the non-minimal entire radial solution case, the asymptotic behavior \eqref{condition-nonminimal} is not enough to make the moving-plane procedure works, we need to obtain more detailed information of the asymptotic behavior of $u$ based on \eqref{condition-nonminimal}.
To know more information of the positive entire solutions with asymptotic behavior \eqref{1.4} near $\infty$, we use a Kelvin type transformation:
\begin{equation}
\label{1.8}
v(y)=|x|^{-\alpha} u(x)-L, \;\;\; y=\frac{x}{|x|^2}
\end{equation}
and make a fundamental estimate for
\begin{equation}
\label{1.9}
W(s):=\Big(\int_{S^{N-1}} w^2 (s, \theta) d \theta \Big)^{\frac{1}{2}},
\end{equation}
where $s=|y|=\frac{1}{r}$, $r=|x|$, $w(s, \theta):=v (s, \theta)-{\overline v} (s)$ and
$${\overline v}(s)=\frac{1}{\omega_{N-1}} \int_{S^{N-1}} v(s, \theta) d \theta, \;\;\; \omega_{N-1}=|S^{N-1}|.$$
The key point is to show that $W(s)$ is Lipschitz continuous, or H\"older continuous in some case, in a neighborhood of $s=0$.

In Sections 2--5, we deal with positive entire solutions $u$ of \eqref{1.1-0} with the asymptotic behavior \eqref{1.4}. In the last section, we deal with positive entire solutions $u$ of \eqref{1.1-0} with \eqref{condition-nonminimal}. In Section 2, we first introduce some preliminary results about the eigenvalues and eigenfunctions of $\Delta^2_{S^{N-1}}$. Then, using the Kelvin-type transformation given in \eqref{1.8} we obtain
the information of $v(y)$ near $y=0$. In Section 3, we derive an important estimate for $W(s)$ (given in \eqref{1.9}) near $s=0$. In Section 4, some estimates for ${\overline v}(s)$ and $v(s, \theta)$ near $s=0$ are obtained. We present the proofs of Theorems \ref{main}, \ref{main-1} and \ref{main-2} in Section 5. Finally, we prove Theorem \ref{main-3} in Section 6. In this paper, we use $C$ to denote a positive constant which may change line by line.

\vskip1cm
\section{Preliminaries}
\setcounter{equation}{0}

In this section, we present some results which will be useful in the following proofs. We use the spherical coordinates $x=(r, \theta)$
as usual. First, let us to show the following lemma (see Lemma 2.1 in \cite{GHZ}).

\begin{lem}
\label{l2.1}
Let $(\lambda, Q(\theta))$ be a pair of eigenvalue and eigenfunction of the equation
\begin{equation}
\label{2.1}
-\Delta_{S^{N-1}} Q=\lambda Q.
\end{equation}
Then $(\lambda^2,Q (\theta))$ is a pair of eigenvalue and eigenfunction of the equation
\begin{equation}
\label{2.2}
\Delta^2_{S^{N-1}} Q=\sigma Q.
\end{equation}
Conversely, if $(\sigma, Q(\theta))$ is a pair of eigenvalue and eigenfunction of \eqref{2.2} with $\sigma \neq 0$, then
$\sigma>0$ and $(\sigma^{1/2}, Q(\theta))$ is a pair of eigenvalue and eigenfunction of \eqref{2.1}.
\end{lem}

It is known from \cite{CH} that for $N \geq 3$, the eigenvalues of the equation \eqref{2.1} are given by
\begin{equation}
\label{2.3}
\lambda_k=k (N+k-2), \;\;\; k \geq 0, \;\; k \in \mathbb{N}
\end{equation}
with multiplicity
\begin{equation}
\label{2.4}
m_k=\frac{(N-2+2k)(N-3+k)!}{k!(N-2)!}.
\end{equation}
Then Lemma \ref{l2.1} implies that the eigenvalues of the equation \eqref{2.2} are $\lambda_k^2$ with the same multiplicity. In particular,
we have
$$ \begin{array}{l} \lambda_0=0, \;\; m_0=1, \;\; Q_1^0 (\theta) \equiv \frac{1}{\sqrt{|S^{N-1}|}},\vspace{2mm}\\
\lambda_1=N-1, \;\; m_1=N, \;\; Q_j^1 (\theta)=\frac{x_j|_{S^{N-1}}}{\|x_j|_{S^{N-1}}\|_{L^2}}, \;\; 1 \leq j \leq N \; (:=m_1),\vspace{2mm}\\
\lambda_2=2N.
\end{array}
$$
Therefore, if $w \in H^2 (S^{N-1})$
is orthogonal to $Q_1^0$, i.e. ${\overline w}=0$, we have
$$\int_{S^{N-1}} |\nabla_\theta w|^2 d \theta \geq (N-1) \int_{S^{N-1}} w^2 d \theta,$$
and
$$\int_{S^{N-1}} |\Delta_\theta w|^2 d \theta \geq (N-1)^2 \int_{S^{N-1}} w^2 d \theta.$$
The boot-strap argument implies that for $1 \leq j \leq m_k$,
\begin{equation}
\label{2.4-1}
\max_{S^{N-1}} |Q_j^k (\theta)| \leq D_k , \;\;\; \max_{S^{N-1}} |(Q_j^k)_\theta (\theta)| \leq E_k,
\end{equation}
where
\begin{equation}
\label{2.4-2}
D_k:=C (1+\lambda_k+\lambda_k^2+\ldots+\lambda_k^\tau), \;\;\; E_k:=C (1+\lambda_k+\lambda_k^2+\ldots+\lambda_k^{\tau_1})
\end{equation}
with $C>0$ being independent of $k$ and $\tau \geq 1$, $\tau_1 \geq 1$ being positive integers such that $2 \tau>N-1$, $2 \tau_1>N$.

In Sections 2-5, we assume that $u \in C^4 (\R^N)$ is a positive entire solution of \eqref{1.1-0} with \eqref{1.4}.
Introducing the Kelvin-type transformation:
\begin{equation}
\label{2.5}
v(y)=|x|^{-\alpha} u(x)-L, \;\;\; y=\frac{x}{|x|^2}, \;\; r=|x|>0,
\end{equation}
we see that $u(x)=u(r, \theta)$, $v(y)=v(s, \theta)$ with $s=|y|=r^{-1}$ and
\begin{eqnarray*}
\Delta_x^2 u &=& \Big[\partial_r^{4}+2(N-1)r^{-1} \partial_r^{3}+(N-1)(N-3) r^{-2} \partial_r^{2}-(N-1)(N-3) r^{-3} \partial_r\\
& &+(8-2N)r^{-4} \Delta_\theta+(2N-6) r^{-3} \Delta_\theta \partial_r+2r^{-2}\Delta_\theta \partial_r^{2}+r^{-4} \Delta_\theta^2\Big] u,
\end{eqnarray*}
with the notations $\partial_r=\frac{\partial}{\partial r}$ and $\partial_r^{m}=\frac{\partial^m}{\partial r^m}$ for $2 \leq m \leq 4$.
Direct calculations imply that
\begin{eqnarray*}
& & \partial_s^4 v-2(N-7+2\alpha) s^{-1}\partial_s^3 v+(N^2+6\alpha N+6\alpha^2-16N-36\alpha+51) s^{-2} \partial_s^2 v \\
& & \;\;\;\;\;-(N-5+2\alpha)(2N\alpha+2\alpha^2-3N-10\alpha+9) s^{-3} \partial_s v\\
& &\;\;\;\;\; + \alpha(\alpha-2)(N+\alpha-2)(N+\alpha-4)s^{-4}(v+L)-2(N-5+2\alpha)s^{-3}\Delta_{\theta}(\partial_s v) \\
& & \;\;\;\;\;+2(N\alpha+\alpha^2-N-4\alpha+4)s^{-4} \Delta_{\theta} v + 2s^{-2} \Delta(\partial_s^2 v)+s^{-4} \Delta^2_{\theta} v\\
& &\;\;=-r^{8-\alpha} \Delta^2_{x}u\,\,= -r^{8-\alpha}u^{-p}\,\,=\,-s^{-8+\alpha(p+1)}(v+L)^{-p}.
\end{eqnarray*}
Since $L=[\alpha (2-\alpha)(N-2+\alpha)(N-4+\alpha)]^{-\frac{1}{p+1}}$ and $\alpha=\frac{4}{p+1}$, we have
\begin{equation}
\label{2.6}
\begin{array}{l} \partial_s^4 v-2(N-7+2\alpha)s^{-1} \partial_s^3 v+(N^2+6\alpha N+6\alpha^2-16N-36\alpha+51) s^{-2} \partial_s^2 v \vspace{1mm}\\
\;\;\;\;-(N-5+2\alpha)(2N\alpha+2\alpha^2-3N-10\alpha+9) s^{-3} \partial_s v   \vspace{1mm}\\
\;\;\;\;-2(N-5+2\alpha) s^{-3} \Delta_{\theta} (\partial_s v)+2(N\alpha+\alpha^2-N-4\alpha+4) s^{-4} \Delta_{\theta} v  \vspace{1mm}\\
\;\;\;\;+ 2s^{-2} \Delta_\theta (\partial_s^2 v)+s^{-4} \Delta^2_{\theta}v -(p+1)s^{-4}L^{-(p+1)} v + s^{-4} f(v)=0,
\end{array}
\end{equation}
where $f(t)=(t+L)^{-p}-L^{-p}+p L^{-(p+1)} t=O(t^2)$ for $t$  near 0.
Note that $f(t)$ is real analytic at $t=0$ and satisfies $f(0)=f'(0)=0$, $f''(0)=p(p+1) L^{-(p+2)}>0$. Therefore, the study of the behavior
of $u$ near $|x|=\infty$ is converted to the study of the behavior of $v$ of the equation \eqref{2.6} near $|y|=0$.

\begin{lem}
\label{l2.2}
Let $u \in C^4(\R^N)$ be a positive entire solution of \eqref{1.1-0} and let $v$ be given in \eqref{2.5}. Suppose that
\begin{equation}
\label{2.7}
|x|^{-\alpha} u(x)-L \to 0 \;\;\; \mbox{as $|x| \to \infty$}.
\end{equation}
Then for any integer $\ell \geq 0$ there exist constants $M=M(u)>0$, $s^*=s^* (u)>0$ such that
\begin{equation}
\label{2.8}
\lim_{|y| \to 0} v(y)=0, \;\;\; |\nabla^\ell v(y)| \leq \frac{M}{s^\ell} \;\;\; \mbox{for $s=|y| \leq s^*$}.
\end{equation}
\end{lem}

\begin{proof} The estimates in \eqref{2.8} follow from \eqref{2.7} by standard elliptic theory. \end{proof}

By Lemma \ref{l2.2}, we are reduced to study solutions of \eqref{2.6} satisfying \eqref{2.8}. Therefore, we will
assume that \eqref{2.8} holds in Sections 2-5.

Define
\begin{equation}
\label{2.9}
w(s, \theta)=v(s, \theta)-{\overline v}(s),
\end{equation}
where
$${\overline v}(s)=\frac{1}{\omega_{N-1}} \int_{S^{N-1}} v(s, \theta) d \theta, \;\;\; \omega_{N-1}=|S^{N-1}|.$$

\begin{lem}
\label{l2.3}
Let $v$ be a solution of \eqref{2.6}. Then ${\overline v}$ and $w$ satisfy
\begin{equation}
\label{2.10}
\begin{split} &{\p_s^4\overline v}-2(N-7+2 \alpha) s^{-1} {\p_s^3\overline v}   \\
&\;\;\;\;\;\;+(N^2+6\alpha N+6\alpha^2-16N-36\alpha+51) s^{-2} {\p_s^2\overline v}   \\
&\;\;\;\;\;\;-(N-5+2\alpha)(2N\alpha+2\alpha^2-3N-10\alpha+9) s^{-3} {\p_s\overline v}   \\
&\;\;\;\;\;\;-(p+1) s^{-4}L^{-(p+1)} {\overline v}+s^{-4} {\overline {f(v)}}=0
\end{split}
\end{equation}
and
\begin{equation}
\label{2.11}
\begin{split}& \partial_s^{4} w-2(N-7+2 \alpha) s^{-1} \partial_s^3 w   \\
&\;\;\;\;\;\;+(N^2+6\alpha N+6\alpha^2-16N-36\alpha+51) s^{-2}\partial_s^2 w  \\
&\;\;\;\;\;\;-(N-5+2\alpha)(2N\alpha+2\alpha^2-3N-10\alpha+9) s^{-3} \partial_s w  \\
&\;\;\;\;\;\;+2(N\alpha+\alpha^2-N-4\alpha+4) s^{-4} \Delta_\theta w-2(N-5+2\alpha) s^{-3} \Delta_\theta (\partial_s w)  \\
&\;\;\;\;\;\;+2 s^{-2} \Delta_\theta (\partial_s^2 w)+s^{-4} \Delta_\theta^2 w-(p+1)s^{-4}  L^{-(p+1)} w + s^{-4} g(w)=0,
\end{split}
\end{equation}
respectively, where
$$g(w):=f(v)-{\overline {f(v)}}=f'(\xi (s, \theta)) w(s, \theta)-{\overline {f'(\xi (s, \theta)) w(s, \theta)}}$$
and
$\xi (s, \theta)$ is between $v(s, \theta)$ and ${\overline v}(s)$.
\end{lem}

\begin{proof} Since $$\overline{\Delta_\theta v}=\frac{1}{\omega_{N-1}}\int_{S^{N-1}}\Delta_\theta v(s,\theta) d \theta =0 ,$$
direct calculations derive \eqref{2.10} and \eqref{2.11}. Moreover, we have
\begin{eqnarray*}
g(w)&=&f(v)-{\overline {f(v)}}=f(v)-f({\overline v})-({\overline {f(v)-f({\overline v})}})\\
&=&f'(\xi (s, \theta)) w(s, \theta)-{\overline {f'(\xi (s, \theta)) w(s, \theta)}}
\end{eqnarray*}
for some $\xi (s, \theta)$ between $v(s, \theta)$ and ${\overline v}(s)$. Where
\begin{equation}
\label{2.12}
f'(\xi (s, \theta))=p L^{-(p+1)}-p [\xi (s, \theta)+L]^{-(p+1)} \geq 0.
\end{equation}
If we define \[\zeta (s):=\max_{\theta \in S^{N-1}} f'(\xi (s, \theta)),\] we see that
$\zeta (s) \to 0$ as $s \to 0$.
\end{proof}

To end this section, we notice that since $w(s, \cdot) \in H^2 (S^{N-1}) \subset L^2 (S^{N-1})$ and ${\overline w}=0$,
\begin{equation}
\label{2.13}
w(s, \theta)=\sum_{k=1}^\infty \sum_{j=1}^{m_k} w_j^k (s) Q_j^k (\theta),
\end{equation}
where $\{Q_1^0 (\theta), Q_1^1 (\theta), \ldots, Q_{m_1}^1 (\theta), Q_1^2 (\theta), Q_2^2 (\theta), \ldots, Q_{m_2}^2 (\theta), Q_1^3 (\theta), \ldots\}$ is the standard normalized basis of $H^2 (S^{N-1})$,
i.e.,
$\int_{S^{N-1}} Q_l^i (\theta) Q_m^j (\theta) d \theta=0$ if $i \neq j$ or $l \neq m$, $\|Q_j^i\|_{L^2(S^{N-1})}=1$
which is consisted by all the eigenfunctions of the operator $-\Delta_{S^{N-1}}$ or $\Delta^2_{S^{N-1}}$ in $H^2 (S^{N-1})$.
Note that $\{Q_1^1 (\theta), \ldots, Q_N^1 (\theta)\}$ is the basis of the eigenspace $H_1$ of $\Delta_{S^{N-1}}^2$ corresponding to the eigenvalue
$(N-1)^2$.

\vskip1cm
\section{ A priori estimate of $W(s)$ for $s$ near 0}
\setcounter{equation}{0}

In this section, we establish some fundamental estimates of $W(s)$ for $s$ near 0, where $W(s)$ is defined by
\begin{equation}
\label{3.1}
W(s)=\Big(\int_{S^{N-1}} w^2 (s, \theta) d \theta \Big)^{\frac{1}{2}}.
\end{equation}
We will see that the Lipschitz continuity of $W(s)$ at the origin is crucial in proving the expansion of $u$ near $\infty$, which can be used
to obtain the symmetry of $u$ by the moving-plane method.

\begin{prop}
\label{p3.1}
For $N \geq 3$, there exist $0<s_0<\min \{1, s^*\}$ $(s^*$ is given in Lemma \ref{l2.2}$)$, $0<\hat{\beta}<1$ and $C>0$ independent of $s$ such that for $s \in (0, s_0)$,
\begin{equation}
\label{3.2}
W(s) \leq \left \{ \begin{array}{ll} C s, \;\; &\mbox{for $N$ and $p$ satisfying \eqref{pN} or $p=7$ and $N=4$ with \eqref{1.6}},\vspace{2mm}\\
C s^{\hat{\beta}}, \;\; & \mbox{for $p \in (\frac{N+2}{6-N},p^*)$ and $N=3,4,5$},
\end{array} \right.
\end{equation}
where $p^*$ is given by \eqref{p*}.
\end{prop}

In fact, $\hat{\beta}=|\beta_3^{(1)}|=|5-N-2\alpha| \in (0,1)$ is given by \eqref{beta} below when $p \in (\frac{N+2}{6-N},p^*)$ and $N=3,4,5$.)

\begin{proof}
Let $Q_j^k (\theta) \; (1 \leq j \leq m_k, k=1,2, \ldots)$ be an eigenfunction of $-\Delta_{S^{N-1}}$ corresponding
to $\lambda_k=k (N+k-2)$. From Lemma \ref{l2.3}, we see that $w_j^k (s)$ satisfies the equation
\begin{equation}
\label{3.3}
\begin{split} & (w_j^k)^{(4)}(s)-2(N-7+2 \alpha) s^{-1} (w_j^k)^{(3)}(s) \\
&\;\;\;+\big(N^2+6\alpha N+6\alpha^2-16N-36\alpha+51-2 \lambda_k\big) s^{-2} (w_j^k)''(s)  \\
&\;\;\;-\big[(N-5+2\alpha)(2N\alpha+2\alpha^2-3N-10\alpha+9)-2(N-5+2 \alpha) \lambda_k \big] s^{-3} (w_j^k)'(s)  \\
&\;\;\;+\big[\lambda_k^2-2\big(N\alpha+\alpha^2-N-4\alpha+4\big) \lambda_k-(p+1) L^{-(p+1)} \big] s^{-4} w_j^k=s^{-4} g_j^k (s),
\end{split}
\end{equation}
where
$$g_j^k (s)=\int_{S^{N-1}} f'(\xi (s, \theta)) w(s, \theta) Q_j^k (\theta) d \theta,$$
which can be controlled by $|g_j^k (s)| \leq  \zeta (s) W(s)$,  here $\zeta (s) \to 0$ and $W(s) \to 0$ as $s \to 0$.

Note that $g_j^k(s)$ and $w_j^k(s)$ are Fourier's coefficients of $f^{\prime}(\xi)w(s,\theta)$ and $w(s,\theta)$ respectively. Moreover,
\begin{equation}
\label{L2-est} \|f^{\prime}(\xi)w(s,\theta)\|_{L^2(S^{N-1})}\le \zeta(s)\|w(s,\theta)\|_{L^2(S^{N-1})}=o_s(1)\|w\|_{L^2(S^{N-1})}
\end{equation}
and
$W (s)=[\sum_{k=1}^{\infty}\sum_{j=1}^{m_k} (w_j^k(s))^2]^{\frac{1}{2}}.$
Therefore, for any $(j,k)$ fixed and $s$ sufficiently small, to estimate $W(s)$, we only need to assume
 \begin{equation}
\label{3.3-3}
|g_j^k (s)|=o_s (1) |w_j^k (s)|.
\end{equation}
In fact, from (\ref{L2-est}), the expression of $w(s,\theta)$ given by (\ref{2.13}) and
\[ f^{\prime}(\xi)w(s,\theta)=\sum_{k=1}^{\infty}\sum_{j=1}^{m_k} g_j^k(s)Q_j^k(\theta),\]
we see that
\[\sum_{k=1}^{\infty}\sum_{j=1}^{m_k} (g_j^k(s))^2=o_s(1)\sum_{k=1}^{\infty}\sum_{j=1}^{m_k} (w_j^k(s))^2.\]
Therefore, there are two cases:
\[  (i)  \quad |g_j^k (s)|=o_s (1) |w_j^k (s)|;\qquad
(ii) \quad |g_j^k (s)|\neq o_s (1) |w_j^k (s)|.
 \]
For any fixed $s \in (0, s^*)$, denote
$$G_s=\{(j,k):1\le j\le m_k, k\ge 1 \,\,\,\mbox{such that (i) holds}\},$$
$$B_s=\{(j,k):1\le j\le m_k, k\ge 1\,\,\, \mbox{such that (ii) holds}\}.$$
We claim that there exists $C>0$ independent of $j, k$ and $s$ such that for any $s\in (0,s^{**})$ and any $(j,k)\in B_s$,
\begin{equation}
\label{inf-est}  |g_j^k (s)|\ge C |w_j^k (s)|,
\end{equation}
where $0<s^{**} \leq s^*$ ($s^*$ is given in Lemma \ref{l2.2}). Suppose not, there exists $c_n \to 0$,
$s_n \to 0$ as $n\to \infty$ and $(j_n,k_n)\in B_{s_n}$ such that
 \[|g_{j_n}^{k_n} (s_n)|\le c_n |w_{j_n}^{k_n} (s_n)|.\]
This implies
\[ |g_{j_n}^{k_n} (s_n)|=o_{s_n} (1) |w_{j_n}^{k_n} (s_n)| \;\;\; \mbox{for $n$ large enough,}\]
which contradicts $(j_n,k_n)\in B_{s_n}$. Therefore, for any $s \in (0, s^{**})$,
\[\sum_{(j,k)\in B}|w_j^k(s)|^2\le C^{-2}\sum_{(j,k)\in B}|g_j^k(s)|^2\le o_s(1)\sum_{k=1}^\infty\sum_{j=1}^{m_k}|w_j^k(s)|^2 .\]
Therefore, without loss of generality, we assume that (\ref{3.3-3}) holds for $0<s\le s^{**}$ , any $k\ge 1$ and $1 \leq j \leq m_k$.

Let $t=-\ln s$, $z_j^k (t)=w_j^k (s)$. Then $z_j^k (t)$ satisfies the equation
\begin{equation}
\label{3.4}
\begin{array}{l}
(z_j^k)^{(4)}(t)+2(N-4+2 \alpha)(z_j^k)^{(3)}(t)\vspace{1mm}\\
\;\;\;\;\;\;+\big[N^2+6\alpha N+6\alpha^2-10N-24\alpha+20-2\lambda_k\big] (z_j^k)''(t) \vspace{1mm}\\
\;\;\;\;\;\;+2(N-4+2\alpha)(N\alpha-N-4\alpha+\alpha^2+2-\lambda_k) (z_j^k)'(t)\vspace{1mm}\\
\;\;\;\;\;\;+\big[\lambda_k^2-2(N\alpha+\alpha^2-N-4\alpha+4) \lambda_k-(p+1) L^{-(p+1)}\big] z_j^k (t)={\tilde g}_j^k (t),
\end{array}
\end{equation}
where ${\tilde g}_j^k (t)=g_j^k (e^{-t})$. We also know from $\zeta(s)$ and $g_j^k(s)$ that
\begin{equation}
\label{3.4-1}
|{\tilde g}_j^k (t)| \leq {\tilde \zeta} (t) {\tilde W}(t),
\end{equation}where
\begin{equation}
\label{3.4-2}
{\tilde \zeta} (t):= \zeta (e^{-t}) \to 0, \;\;\;\; \mbox{and} \;\; {\tilde W}(t):=W(e^{-t}) \to 0 \;\; \;\mbox{as $t \to \infty$}.
\end{equation}

The corresponding characteristic polynomial of \eqref{3.4} is
\begin{equation}
\label{3.5}
\begin{array}{l}
\beta^4+2(N-4+2 \alpha) \beta^3+\big(N^2+6\alpha N+6\alpha^2-10N-24\alpha+20\vspace{1mm}\\
\;\;\;\;\;\;
-2\lambda_k\big) \beta^2 +2(N-4+2\alpha)(N\alpha-N-4\alpha+\alpha^2+2-\lambda_k)\beta \vspace{1mm}\\
\;\;\;\;\;\;+\lambda_k^2-2(N\alpha+\alpha^2-N-4\alpha+4) \lambda_k-(p+1) L^{-(p+1)}=0
\end{array}
\end{equation}
and using the MATLAB, the four roots of \eqref{3.5} are given by
\begin{equation}
\label{3.6}
\left \{\begin{split}
&\beta_1^{(k)}=\frac{1}{2}\Big(4-N-2\alpha+\sqrt{4+(N-2+2k)^2+4\sqrt{\rho_k}}\,\Big),\\
&\beta_2^{(k)}=\frac{1}{2}\Big(4-N-2\alpha-\sqrt{4+(N-2+2k)^2+4\sqrt{\rho_k}}\,\Big),\\
&\beta_3^{(k)}=\frac{1}{2}\Big(4-N-2\alpha+\sqrt{4+(N-2+2k)^2-4\sqrt{\rho_k}}\,\Big),\\
&\beta_4^{(k)}=\frac{1}{2}\Big(4-N-2\alpha-\sqrt{4+(N-2+2k)^2-4\sqrt{\rho_k}}\,\Big),
\end{split} \right.
\end{equation}
where
\begin{equation}
\label{3.7}
\rho_k=(N-2+2k)^2+p\alpha(2-\alpha)(N-2+\alpha)(N-4+\alpha).
\end{equation}

We first analyze the four roots $\beta^{(k)}_j, j=1,2,3,4$ for $k=1, 2, \ldots$. Note that
\begin{equation}
\label{fact-claim2}
\alpha\in (1,2),\;\; \mbox{for $N=3$, $1<p<3$;}\quad
\alpha\in (0,2),\;\; \mbox{for $N\ge 4$, $p>1$}.
\end{equation}
Then, we see from \eqref{3.7} that $\rho_k>0$ for $k=1,2, \ldots$.
Set \[T_k:=[4+(N-2+2k)^2]^2-16 \rho_k.\]
Noticing that $p \alpha=4-\alpha$, we have
$$(2-\alpha)(N+\alpha-4)\le (N-2)^2,\quad (4-\alpha)(N+\alpha-2)\le (N+2)^2.$$
Thus
\begin{eqnarray*}
T_k&=&[(N-2+2k)^2-4]^2-16(4-\alpha)(2-\alpha)(N+\alpha-2)(N+\alpha-4)\\
&\ge& [N^2-4]^2-(N-2)^2 (N+2)^2 \ge 0
\end{eqnarray*}
for $N=3$ and $p \in (1,3)$; $N \ge 4$ and $p \ge 1$ and any $k=1,2, \ldots$.
This indicates that $\beta_j^{(k)}$ are real numbers for $k=1,2, \ldots$.
Therefore, taking into account the expressions in \eqref{3.6} of $\beta_j^{(k)}$,
we have demonstrated the following statement.\medskip

\noindent{\it Claim 1.} For any $k \geq 1$; $N=3$ and $p\in (1,3)$; $N\ge 4$ and $p>1$, the roots $\beta_j^{(k)}\; (j=1,2,3,4)$ are real numbers.
Moreover,
$$\beta_2^{(k)}<\beta_4^{(k)}\le\beta_3^{(k)}<\beta_1^{(k)}.$$

\begin{rem}
\label{rem3.2}
For $N=3$ and $p\in (1,3)$; $N \ge 4$ and $p>1$, noticing that
\begin{equation}
\label{fact-rem3.2}
4-2\alpha-N<0\quad \mbox{and} \quad 4+(N+2k-2)^2-(4-2\alpha-N)^2>0,
\end{equation}
we find from \eqref{3.6} that
\begin{equation}
\label{relation-rem3.2}
\beta_2^{(k)}<\beta_4^{(k)}<  0 <\beta_1^{(k)} \quad \mbox{for any $k\ge 1$}.
\end{equation}
\end{rem}

We now determine the sign of $\beta_3^{(k)}$ for any $k \ge 2$.\medskip

\noindent{\it Claim 2.} For any $k \ge 2$, $\beta_3^{(k)}>0$ when $N=3$ and $p \in (1,3)$; $N \ge 4$ and $p>1$.\medskip

In fact, by \eqref{fact-rem3.2}, we see that $\beta_3^{(k)}>0$ is equivalent to
$$ {\hat T}_k:=\big[4+(N+2k-2)^2-(4-2\alpha-N)^2 \big]^2-16\rho_k > 0,$$
for any $k \ge 2$ when $N=3$ and $p\in (1,3)$; $N \ge 4$ and $p>1$.
Writing ${\hat T}_k$ as
\begin{eqnarray*}
{\hat T}_k&=&16(k-\alpha)(k+2-\alpha)(N+\alpha-2+k)(N+\alpha-4+k)\\
& &\;\;\;\; -16(2-\alpha)(4-\alpha)(N+\alpha-2)(N+\alpha-4).
\end{eqnarray*}
Therefore Claim 2 follows since we have
\begin{eqnarray*}
& &k-\alpha\ge 2-\alpha > 0, \;\;\;\;\;\;\;\;\;\;\;  N+\alpha-2+k\ge N+\alpha-2>0,\\
& & k+2-\alpha\ge 4-\alpha>0, \;\;\;\;  N+\alpha-4+k\ge N+\alpha-4>0.
\end{eqnarray*}

For the root $\beta_4^{(k)}$, we have the following assertion. \medskip

\noindent{\it Claim 3.} For any $k \geq 2$, $\beta_4^{(k)}<-1$ for $N =3$ and $p\in (1,3)$; $N \ge 4$ and $p>1$.

By the expression of $\beta_4^{(k)}$, we have
\begin{equation}
\label{3.6-1}
\beta_4^{(k)}+1=\frac{1}{2} \Big[6-N-2\alpha -\sqrt{4+(N+2k-2)^2-4 \sqrt{\rho_k}}\,\Big].
\end{equation}

Obviously,
$\beta_4^{(k)}+1<0$ when $N \ge 6$ and $p>1$, and even in the case of $N \in \{3,4,5\}$ and $6-N-2\alpha \leq 0$, i.e.
$p \in (1,\frac{N+2}{6-N}]$. On the other cases, we see that $\beta_4^{(k)}+1<0$ is equivalent to
$$0 < 6-N-2\alpha < \sqrt{4+(N+2k-2)^2-4 \sqrt{\rho_k}}. $$
So
to obtain our claim, it's sufficient to show that
$$ {\tilde T}_k:=\Big[4+(N+2k-2)^2-(6-N -2\alpha)^2\Big]^2-16 \rho_k > 0.$$
Since
\begin{eqnarray*}
{\tilde T}_k&=&16(k+1-\alpha)(k+3-\alpha)(N+\alpha+k-3)(N+\alpha+k-5) \\
& & \;\;\;\;\;-16(2-\alpha)(4-\alpha)(N+\alpha-2)(N+\alpha-4).
\end{eqnarray*}
We find again ${\tilde T}_k>0$ for any $k \ge 2$ and $p \in (\frac{5}{3},3)$ with $N=3$; $p>\frac{N+2}{6-N}$ with $N= 4, 5$ from the facts
\begin{eqnarray*}
&& k+1-\alpha \ge 2-\alpha > 0,\;\;\;\;\;\;\;\; N+\alpha+k-3\ge N+\alpha-2>0,\\
&&k+3-\alpha\ge 4-\alpha>0, \;\;\;\;\;\;\;\; N+\alpha+k-5\ge N+\alpha-4>0.
\end{eqnarray*}
Consequently, the Claim 3 is derived from all these arguments.

\begin{rem}
\label{r3.2}
It follows from Claims 1-3 that for $N =3$ and $p\in (1,3)$; $N\ge 4$ and $p>1$; any $k\ge 2$,
\begin{equation}
\label{3.13}
\beta_2^{(k)}<\beta_4^{(k)}<-1<0<\beta_3^{(k)}<\beta_1^{(k)}.
\end{equation}
Moreover, we deduce from the expressions of $\beta_j^{(k)}$ that
\[\begin{array}{ll}\beta_2^{(k+1)}<\beta_2^{(k)}<0,\quad &\beta_4^{(k+1)}<\beta_4^{(k)}<-1,\vspace{1mm}\\
\beta_3^{(k+1)}>\beta_3^{(k)}>0,& \beta_1^{(k+1)}>\beta_1^{(k)}>0
\end{array}\]
and
\[\begin{array}{ll}
\beta_2^{(k+1)}-\beta_2^{(k)}\to -1 \;\; & \mbox{as $k\to \infty$},\vspace{1mm}\\
\beta_4^{(k+1)}-\beta_4^{(k)}\to -1 \;\; & \mbox{as $k\to \infty$},\vspace{1mm}\\
\beta_3^{(k+1)}-\beta_3^{(k)} \to 1 \;\; &\mbox{as $k \to \infty$}.
\end{array}\]
\end{rem}

Now we investigate the details of $\beta_j^{(1)}, \; j=1,2,3,4$. Recalling that $p^*$ is given by \eqref{p*}, we have:\medskip

\noindent{\it Claim 4.}  The following inequalities hold for $k=1$:
\begin{eqnarray*}
\beta_2^{(1)}<\beta_4^{(1)}\le\beta_3^{(1)}=-1<0<\beta_1^{(1)}, &&\mbox{for $\left\{
                                                           \begin{array}{ll}  p \in (1,\frac{N+2}{6-N}],& N=3,4,5;\\
                                                            p \in (1,\infty), & N \geq 6;
                                                           \end{array} \right.$} \\
\beta_2^{(1)}<\beta_4^{(1)}=-1<\beta_3^{(1)}<0<\beta_1^{(1)}, &&\mbox{for $p\in (\frac{N+2}{6-N},p^*)$ and $N =3,4,5$};\\
\beta_2^{(1)}<\beta_4^{(1)}=-1<\beta_3^{(1)}=0<\beta_1^{(1)}, &&\mbox{for $p=7$ and $N =4$};\\
 \beta_2^{(1)}<\beta_4^{(1)}=-1<0<\beta_3^{(1)}<\beta_1^{(1)}, &&\mbox{for $p \in (7,\infty)$ and $N =4$}.
 \end{eqnarray*}

For $k=1$, from the expressions of
$$\rho_1=N^2 + (4-\alpha)(2-\alpha)(N-2+\alpha)(N-4+\alpha),$$
a direct calculation shows that
$$\sqrt{4+(N)^2-4 \sqrt{\rho_1}}=|N-6 +2 \alpha|$$
and therefore
\begin{eqnarray*}
& & \beta_3^{(1)}=\frac{1}{2} \Big[4-N-2\alpha -|N-6 +2\alpha| \Big],\\
& &\beta_4^{(1)}=\frac{1}{2} \Big[4-N-2\alpha +|N-6 +2\alpha| \Big].
\end{eqnarray*}
As before, we have obviously $N-6 +2\alpha >0$ when $N \ge 6$ or $N \in \{3,4,5\}$ and $p\in (1, \frac{N+2}{6-N})$, which implies
$$\beta_4^{(1)}=5-N- 2\alpha <-1=\beta_3^{(1)}.$$
This combining with Claim 1 and \eqref{relation-rem3.2} yields that
$$\beta_2^{(1)}<\beta_4^{(1)}<\beta_3^{(1)}=-1<0<\beta_1^{(1)}.$$
We obtain also for $p\in [\frac{N+2}{6-N},p^*)$ and $N=3$, $4$ or $5$,
$$\beta_4^{(1)}=-1 \leq 5-N-2\alpha =\beta_3^{(1)}<0;$$
and when $N=4$, $p=p^*=\frac{N+3}{5-N}=7$,
$$\beta_4^{(1)}=-1<0=\beta_3^{(1)};$$
when $N=4$, $p>\frac{N+3}{5-N}=7$,
$$\beta_4^{(1)}=-1<0<5-N-2\alpha =\beta_3^{(1)}<1.$$
Combining with Claim 1 and Remark \ref{rem3.2}, we prove that Claim 4 holds.

We continue the proof of Proposition \ref{p3.1}.

For any $k \geq 2$, from the equation satisfied by $z_j^k$ and the ODE theory, we see that, for any $T > -\ln s^{**}$, there exist constants $A_{j,i}^k$, $B_i^k$ $(i=1,2,3,4)$ such that for $t>T$,
\begin{equation}
\label{3.14}
z_j^k (t)=\sum_{i=1}^4 \Big[A_{j,i}^k e^{\beta_i^{(k)}t}+B_i^k \int_T^t e^{\beta_i^{(k)}(t-s)} {\tilde g}_j^k(s) d s \Big],
\end{equation}
where $A_{j,i}^k \; (i=1,2,3,4)$ depend on $T$ and $\beta_i^{(k)}$, but $B_i^k \; (i=1,2,3,4)$ depend only on $\beta_i^{(k)}$. More precisely, the detailed calculations show that
$$A_{j,1}^k=\frac{F_{j,1}^k (T)}{(\beta_1^{(k)}- \beta_2^{(k)})(\beta_1^{(k)}-\beta_3^{(k)})(\beta_1^{(k)}-\beta_4^{(k)})} e^{-\beta_1^{(k)} T},$$
$$A_{j,2}^k=\Big[\frac{F_{j,1}^k (T)}{(\beta_2^{(k)}- \beta_1^{(k)})(\beta_2^{(k)}-\beta_3^{(k)})(\beta_2^{(k)}-\beta_4^{(k)})}+\frac{F_{j,2}^k (T)}{(\beta_2^{(k)}-\beta_3^{(k)})(\beta_2^{(k)}-\beta_4^{(k)})} \Big] e^{-\beta_2^{(k)} T},$$
\begin{eqnarray*}
A_{j,3}^k &=& \Big[\frac{F_{j,1}^k (T)}{(\beta_3^{(k)}- \beta_1^{(k)})(\beta_3^{(k)}-\beta_2^{(k)})(\beta_3^{(k)}-\beta_4^{(k)})}\\
& &\;\;\;\;\;\;\;+\frac{F_{j,2}^k (T)}{(\beta_3^{(k)}-\beta_2^{(k)})(\beta_3^{(k)}-\beta_4^{(k)})}+\frac{F_{j,3}^k (T)}{(\beta_3^{(k)}-\beta_4^{(k)})} \Big] e^{-\beta_3^{(k)} T},
\end{eqnarray*}
\begin{eqnarray*}
A_{j,4}^k&=& \Big[\frac{F_{j,1}^k (T)}{(\beta_4^{(k)}- \beta_1^{(k)})(\beta_4^{(k)}-\beta_2^{(k)})(\beta_4^{(k)}-\beta_3^{(k)})}\\
& &\;\;\;\;\;\;\;+\frac{F_{j,2}^k (T)}{(\beta_4^{(k)}-\beta_2^{(k)})(\beta_4^{(k)}-\beta_3^{(k)})}+\frac{F_{j,3}^k (T)}{(\beta_4^{(k)}-\beta_3^{(k)})}+z_j^k (T) \Big] e^{-\beta_4^{(k)} T},
\end{eqnarray*}
where
$$F_{j,1}^k (T)=(\partial_t-\beta_2^{(k)})(\partial_t-\beta_3^{(k)})(\partial_t-\beta_4^{(k)}) z_j^k (T),$$
$$F_{j,2}^k (T)=(\partial_t-\beta_3^{(k)})(\partial_t-\beta_4^{(k)}) z_j^k (T),$$
$$F_{j,3}^k (T)=(\partial_t-\beta_4^{(k)}) z_j^k (T)$$
and
$$B_{i}^k=\prod_{j \ne i} \frac{1}{\beta_i^{(k)}- \beta_j^{(k)}}, \;\; \;\forall i \in \{1,2,3,4\}.$$
Since $w(s,\cdot) \to 0$ as $s \to 0^{+}$, we have $z_j^k (t) \to 0$ as $t \to \infty$. Moreover, ${\tilde g}_j^k (t) \to 0$ as $t \to \infty$. It follows from $\beta_1^{(k)}>\beta_3^{(k)}>0$ for $k \geq 2$; $N=3$ and $p\in(1,3)$; $N \geq 4$ and $p>1$ that
$$\int_t^{\infty} e^{\beta_1^{(k)}(t-s)} {\tilde g}_j^k (s) d s \to 0, \;\;\;\; \int_t^{\infty} e^{\beta_3^{(k)}(t-s)} {\tilde g}_j^k (s) d s \to 0 \;\; \mbox{as $t\to \infty$}.$$
By means of $\int_T^t=\int_T^{\infty}-\int_t^{\infty}$, we rewrite $z_j^k (t)$ in the following form:
\begin{eqnarray*}
z_j^k (t)&=& M_{j,1}^k e^{\beta_1^{(k)}t}+M_{j,3}^k e^{\beta_3^{(k)}t}+ A_{j,2}^k e^{\beta_2^{(k)}t}+ A_{j,4}^k e^{\beta_4^{(k)}t}\\
& &\;\;-B_1^k \int_t^{\infty} e^{\beta_1^{(k)}(t-s)} {\tilde g}_j^k (s) d s-B_3^k \int_t^{\infty} e^{\beta_3^{(k)}(t-s)} {\tilde g}_j^k (s) d s\\
& &\;\;+B_2^k \int_T^t e^{\beta_2^{(k)}(t-s)} {\tilde g}_j^k (s) d s+B_4^k \int_T^t e^{\beta_4^{(k)}(t-s)} {\tilde g}_j^k (s) d s,
\end{eqnarray*}
where
$$M_{j,1}^k=A_{j,1}^k+B_1^k \int_T^\infty e^{-\tau\beta_1^{(k)}} {\tilde g}_j^k (\tau) d \tau, \;\;\; M_{j,3}^k=A_{j,3}^k+B_3^k \int_T^\infty e^{-\tau\beta_3^{(k)}} {\tilde g}_j^k (\tau) d \tau.$$
The fact that $z_j^k (t) \to 0$ as $t \to \infty$ implies $M_{j,1}^k=M_{j,3}^k=0$. Therefore,
\begin{eqnarray}
\label{3.14-0}
z_j^k (t) &=&A_{j,2}^k e^{\beta_2^{(k)} T}  e^{\beta_2^{(k)} (t-T)}+A_{j,4}^k e^{\beta_4^{(k)} T} e^{\beta_4^{(k)} (t-T)} \nonumber \\
& &\;\;\;\;\;-B_{1}^k \int_t^\infty e^{\beta_1^{(k)} (t-\tau)} {\tilde g}_j^k (\tau) d \tau -B_{3}^k \int_t^\infty e^{\beta_3^{(k)} (t-\tau)} {\tilde g}_j^k (\tau) d \tau \nonumber \\
& &\;\;\;\;\;+B_{2}^k \int_T^t e^{\beta_2^{(k)} (t-\tau)} {\tilde g}_j^k (\tau) d \tau+B_{4}^k \int_T^t e^{\beta_4^{(k)} (t-\tau)} {\tilde g}_j^k (\tau) d \tau.
\end{eqnarray}
We now establish the estimate of $z_j^k(t)$ with $k \geq 1, \; 1 \leq j \leq m_1$.
We start with $k \geq 2$, $1 \leq j \leq m_k$ and claim that
\begin{equation}
\label{3.14-1}
|z_j^k (t)|=O(k e^{\beta_4^{(k)} (t-T)})
\end{equation}
for $t>T$. For any fixed $(k,j)$, if $z_j^k(t)\equiv 0$, this is trivial.
Assume that $z_j^k (t) \not \equiv 0$ for $t \in [T, \infty)$ in the following, it is known from \eqref{3.3-3} that
\begin{equation}
\label{3.14-2}
|{\tilde g}_j^k (t)|=o_t (1) |z_j^k (t)| \;\;\; \mbox{for $t \in (T, \infty)$}.
\end{equation}

It follows from Lemma \ref{l2.2} and \eqref{3.14-0} that, for $t \in (T, \infty)$,
\begin{eqnarray}
\label{3.14-3}
|z_j^k (t)| &\leq& O \Big(k e^{\beta_4^{(k)} (t-T)} \Big)+ C \int_T^t  e^{\beta_4^{(k)} (t-\tau)}  o_\tau (1) |z_j^k (\tau)| d \tau \nonumber \\
& &\;\;\;\;\; +C \int_t^\infty  e^{\beta_3^{(k)} (t-\tau)} o_\tau (1) |z_j^k (\tau)| d \tau.
\end{eqnarray}
Note that
$$e^{\beta_1^{(k)} (t-\tau)} \leq e^{\beta_3^{(k)} (t-\tau)}, \;\; \mbox{and}\;\; e^{\beta_2^{(k)} (t-\tau)} \leq e^{\beta_4^{(k)} (t-\tau)} \;\; \mbox{for $\tau \leq t$}.$$
Note also that
for $\ell=1,3$ and any fixed $t>T$,
$$\Big|\int_t^\infty e^{\beta_\ell^{(k)} (t-\tau)} {\tilde g}_j^k (\tau) d \tau \Big| \leq \int_t^\infty e^{\beta_\ell^{(k)} (t-\tau)} o_\tau (1) |z_j^k (\tau)| d \tau$$
and for $\ell=2,4$,
$$ \Big| \int_T^t e^{\beta_\ell^{(k)} (t-\tau)} {\tilde g}_j^k (\tau) d \tau \Big| \leq \int_T^t e^{\beta_\ell^{(k)} (t-\tau)} o_\tau (1) |z_j^k (\tau)| d \tau.$$

It follows from \eqref{3.14-3} and arguments similar to those
in \cite{GHZ} that
\begin{equation}
\label{3.14-4}
|z_j^k (t)|=O(k e^{\beta_4^{(k)} (t-T)})
\end{equation}
for $t \in (T, \infty)$.
This implies that our claim \eqref{3.14-1} holds for $z_j^k (t) \not \equiv 0$.
Therefore, our claim \eqref{3.14-1} holds.

We now establish the estimate of $z_j^1(t)$ with $1 \leq j \leq m_1$.

We first consider the estimate for $N=4$, which can be split to four cases:
(i) $p \in (1,3]$; (ii) $p \in (3,7)$; (iii) $p\in (7,\infty)$ and (iv) $p=7$.

For the case (i), it is known from Claim 4 that $\beta_2^{(1)}<\beta_4^{(1)}\le -1=\beta_3^{(1)}<0<\beta_1^{(1)}$. The fact $z_j^1 (t) \to 0$ as $t\to \infty$ implies that $z_j^1 (t)$ can be written in the form
\begin{eqnarray*}
z_j^1 (t) &=& A_{j,2}^1 e^{\beta_2^{(1)}t}+ A_{j,3}^1 e^{-t}+A_{j,4}^1 e^{\beta_4^{(1)}t} \\
& &\;\;\;\;-B_1^1 \int_t^{\infty} e^{\beta_1^{(1)}(t-s)} {\tilde g}_j^1 (s) d s+B_2^1 \int_T^t e^{\beta_2^{(1)}(t-s)} {\tilde g}_j^1 (s) d s\\
& &\;\;\;\;+B_3^1 \int_T^t e^{-(t-s)} {\tilde g}_j^1 (s) d s+B_4^1 \int_T^t e^{\beta_4^{(1)}(t-s)} {\tilde g}_j^1 (s) d s.
\end{eqnarray*}
Arguments similar to those in the proof of \eqref{3.14-1} imply that, for $1 \leq j \leq m_1$ and $t>T$,
\begin{equation}
\label{3.16}
|z_j^1 (t)|=O(e^{-(t-T)}).
\end{equation}

For the case (ii), we see from Claim 4 that $\beta_2^{(1)}<\beta_4^{(1)}=-1<\beta_3^{(1)}<0<\beta_1^{(1)}$. Therefore,
\begin{eqnarray*}
z_j^1 (t)&=& A_{j,2}^1 e^{\beta_2^{(1)}t}+ A_{j,3}^1 e^{\beta_3^{(1)}t}+ A_{j,4}^1 e^{-t} \\
& &\;\;\;\;-B_1^1 \int_t^{\infty} e^{\beta_1^{(1)}(t-s)} {\tilde g}_j^1 (s) d s + B_2^1 \int_T^t e^{\beta_2^{(1)}(t-s)} {\tilde g}_j^1 (s) d s \\
& &\;\;\;\;+B_3^1 \int_T^{t} e^{\beta_3^{(1)}(t-s)} {\tilde g}_j^1 (s) d s +B_4^1 \int_T^t e^{-(t-s)} {\tilde g}_j^1 (s) d s.
\end{eqnarray*}
Similarly, we have that, for $1 \leq j \leq m_1$ and $t>T$,
\begin{equation}
\label{3.16-1}
|z_j^1 (t)|= O(e^{\beta_3^{(1)}(t-T)}).
\end{equation}

For the case (iii), Claim 4 shows us that $\beta_2^{(1)}<\beta_4^{(1)}=-1<0<\beta_3^{(1)}<\beta_1^{(1)}$.
Then,
\begin{eqnarray*}
z_j^1 (t) &=& A_{j,2}^1 e^{\beta_2^{(1)}t}+ A_{j,4}^1  e^{-t} \\
& &\;\;\;\;-B_1^1 \int_t^{\infty} e^{\beta_1^{(1)}(t-s)} {\tilde g}_j^1 (s) d s-B_3^1 \int_t^{\infty} e^{\beta_3^{(1)}(t-s)} {\tilde g}_j^1 (s) d s\\
& &\;\;\;\;+B_2^1 \int_T^t e^{\beta_2^{(1)}(t-s)} {\tilde g}_j^1 (s) d s+B_4^1 \int_T^t e^{-(t-s)} {\tilde g}_j^1 (s) d s.
\end{eqnarray*}
By the method analogous to that used above, for $1 \leq j \leq m_1$ and $t>T$, we get
\begin{equation}
\label{3.16-2}
|z_j^1 (t)|=O(e^{-(t-T)}).
\end{equation}

For the case (iv), we know that $\beta_2^{(1)}<\beta_4^{(1)}=-1<\beta_3^{(1)}=0<\beta_1^{(1)}$. Then
\begin{eqnarray*}
z_j^1 (t)&=& A_{j,2}^1 e^{\beta_2^{(1)}t}+ A_{j,4}^1 e^{-t} \\
& &\;\;\;\;-B_1^1 \int_t^{\infty} e^{\beta_1^{(1)}(t-s)} {\tilde g}_j^1 (s) d s-B_3^1 \int_t^{\infty}  {\tilde g}_j^1 (s) d s \\
& &\;\;\;\;+B_2^1 \int_T^t e^{\beta_2^{(1)}(t-s)} {\tilde g}_j^1 (s) d s+B_4^1 \int_T^t e^{-(t-s)} {\tilde g}_j^1 (s) d s.
\end{eqnarray*}
Similarly, we have that, for $1 \leq j \leq m_1$ and $t>T$,
\begin{equation}
\label{3.16-3} |z_j^1 (t)|=O(e^{-(t-T)})+C\int_t^{\infty}
|o_s(1)z_j^1 (s)| d s,
\end{equation}
where $C>0$ which depends only on $B^1_1$ and $B^1_3$ but
independent of $T$. Let
$$ K(t)=\int_t^{\infty} |z_j^1(s)| d s.$$
Then we can obtain that $K (t)$ is bounded provided that the
condition in (\ref{1.6}) holds. In fact, it follows from \eqref{1.6}
that there is $0<\varepsilon_0<\frac{1}{10}$ such that, for $s$ near
0,
$$|w(s,\theta)|^2 \le C s^{2\varepsilon_0}.$$
Consequently,
\begin{eqnarray*}
K(t)&=&\int_0^s \zeta^{-1} |w_1(\zeta)| d\zeta \leq \int_0^s \zeta^{-1} \Big(\int_{S^{N-1}}w^2(\zeta,\theta) d\theta \Big)^{\frac{1}{2}} d\zeta \\
&\le& C \int_0^s \zeta^{\varepsilon_0-1} d\zeta = \frac{C}{\varepsilon_0} e^{-\varepsilon_0 t} <\infty,
\end{eqnarray*}
which implies that
for any $0<\epsilon<\varepsilon_0/C$,
$$\lim_{t \to \infty} e^{C \epsilon t} K(t)=0.$$
On the other hand, it follow from the definition of $K(t)$ that for
$t$ sufficiently large,
$$-K' (t)=|z_j^1(t)| \leq O(e^{-(t-T)})+C \epsilon K(t).$$
We can easily see that
$$K(t)=O(e^{-(t-T)}).$$
This and \eqref{3.16-3} imply
\begin{equation}
\label{} |z_j^1(t)|=O(e^{-(t-T)}).
\end{equation}

Now, let ${\hat W}(t)=\Sigma_{k=1}^\infty \Sigma_{j=1}^{m_k} |z_j^k (t)|$. Then ${\tilde W}(t)=\Big(\Sigma_{k=1}^\infty \Sigma_{j=1}^{m_k} (z_j^k)^2 (t) \Big)^{1/2} \leq {\hat W}(t)$. For the cases (i), (iii) and (iv), we see from \eqref{3.14-1}, \eqref{3.16}, \eqref{3.16-2}, \eqref{3.16-3} that
\begin{equation}
\label{3.16-4} {\tilde W}(t) \leq {\hat W}(t) \leq O \Big(e^{-t}
\Big)+O \Big(\sum_{k=2}^\infty k m_k e^{\beta_4^{(k)} (t-T)} \Big).
\end{equation}
Let $T^*=10 T$. We obtain that, for $t>T^*$,
\begin{equation}
\label{3.17} \sum_{k=2}^\infty k m_k e^{\beta_4^{(k)}
(t-T)}=O(e^{\beta_4^{(2)} (t-T)}).
\end{equation}
To see \eqref{3.17}, we notice that, for any $t>T^*$ (we may enlarge $T^*$),
$$\lim_{k \to \infty} \frac{(k+1) m_{k+1}  e^{\beta_4^{(k+1)} (t-T)}}{k m_k e^{\beta_4^{(k)} (t-T)}}
=e^{-(t-T)} \lim_{k \to \infty} \frac{(k+1) m_{k+1}}{k m_k}=e^{-(t-T)}<\frac{1}{2}.$$
Since $\beta_4^{(2)}<-1$, we easily have that, for $t>T_*$,
\begin{equation}
\label{3.18}
{\tilde W} (t)=O(e^{-t}).
\end{equation}
Let $s_0=e^{-T^*}$. We see from \eqref{3.18} that there exists $C>0$ such that,
\begin{equation}
\label{3.19}
W(s) \leq C s \qquad \mbox{for $0<s<s_0$}.
\end{equation}
Arguments similar to the above imply that we can obtain
\begin{equation}
\label{3.20}
{\tilde W}(t)=O \left(e^{\beta_3^{(1)} t} \right)
\end{equation}
for the case (ii). Note that $\beta_3^{(1)} \in (-1, 0)$ in this case.

For $N=3$ or $N\ge 5$, processing the same procedure as above, we can obtain
\begin{equation}
\label{3.22}
{\tilde W}(t)=\left\{\begin{array}{ll}O(e^{-t}), \;\; &\mbox{for} \left\{\begin{array}{lll}
                                                            p\in (1, \frac{N+2}{6-N}] & & \mbox{when $N=3$ or $5$,}\\
                                                            p\in (1,3]\cup (7,\infty) &&\mbox{when $N=4$},\\
                                                            p=7, && \mbox{when $N=4$ with \eqref{1.6}},\\
                                                            p\in (1, \infty) && \mbox{when $N\ge 6$};\end{array}\right.\vspace{1mm} \\
               O(e^{\beta_3^{(1)}t}), \;\; & \mbox{for $p \in (\frac{N+2}{6-N},p^*)$} \hspace{21mm} \mbox{when $N=3$, $4$ or $5$},
\end{array} \right.
\end{equation}
where $p^*$ is given by \eqref{p*}, $\beta_3^{(1)}=5-N-2\alpha \in (-1,0)$ when $p \in (\frac{N+2}{6-N},p^*)$ and $N=3$, $4$ or $5$.
Choosing
\begin{equation}
\label{beta}
\hat{\beta}=|\beta_3^{(1)}|=N+ 2\alpha-5,\quad \mbox{$p \in (\frac{N+2}{6-N},p^*)$ and $N=3$, $4$ or $5$,}
\end{equation}
we see that $0<\hat{\beta}<1$ in this case. Since $W(s)={\tilde W}(t)$ and $t=-\ln s$, we obtain the conclusions of Proposition \ref{p3.1} from \eqref{3.22}. This completes the proof of Proposition
\ref{p3.1}.
\end{proof}

\section{Estimates for ${\overline v}(s)$, $v(s,\theta)$ near $s=0$ and expansions of $u (r, \theta)$ near $r=\infty$}
\setcounter{equation}{0}

This section is devoted to establish some estimates for $\overline{v}(s)$ and $v(s,\theta)$ near $s=0$ which enable us to obtain expansions of positive entire solutions $u (r, \theta)$ of \eqref{1.1-0} at $r=\infty$.

We begin our analysis by recalling the equation satisfied by $\overline{v}(s)$. From Lemma \ref{l2.3},
we see that
\begin{eqnarray*}
& & {\overline v}^{(4)} -2(N-7+2 \alpha) s^{-1} {\overline v}^{(3)} +\big(N^2+6\alpha N+6\alpha^2-16N-36\alpha+51) s^{-2}
                    {\overline v}'' \\
& &\;\;\;\;\;\;\;-(N-5+2\alpha)(2N\alpha+2\alpha^2-3N-10\alpha+9) s^{-3} {\overline v}'- (p+1)s^{-4} L^{-(p+1)} {\overline v} \\
& & \;\;\;\;=s^{-4} \left[{f(\overline {v})-\overline {f(v)}}\right]-s^{-4} {f(\overline {v})}
\end{eqnarray*}
and
\begin{eqnarray*}
& &\left|f(\overline {v})-{\overline {f(v)}}\right| \leq \frac{1}{\omega_{N-1}}\int_{S^{N-1}}|f(v)-f(\overline{v})| d \theta \\
& &\;\;\;\leq o_s \Big[ \Big(\int_{S^{N-1}}w^2 \Big)^{\frac{1}{2}} \Big] \\
& &\;\;\;=\left \{ \begin{array}{ll} o(s), \;\;&\mbox{for $N$ and $p$ satisfying \eqref{pN} or $p=7$, $N =4$ with \eqref{1.6},}\vspace{1mm}\\
o (s^{\hat{\beta}}), \;\; & \mbox{for $p \in (\frac{N+2}{6-N},p^*)$ and $N=3$, $4$ or $5$}, \end{array} \right.
\end{eqnarray*}
where $\hat{\beta}$ is given in \eqref{beta}.

Let $t=-\ln s$ and $\overline{z}(t)=\overline{v}(s)$. Then $\overline{z}(t)$ satisfies
\begin{equation}
\label{4.1}
\begin{array}{l} \overline{z}^{(4)} +2(N+2\alpha-4)\overline{z}^{(3)}+(N^2+6N\alpha+6\alpha^2-10N-24\alpha+20) \overline{z}''\vspace{1mm}\\
 \;\;\;\;\;\;+2(N+2\alpha-4)(N \alpha +\alpha^2-N -4\alpha+2)\overline{z}'-(p+1) L^{-(p+1)}\overline{z} \vspace{1mm}\\
\;\;\; =\left \{ \begin{array}{ll} -f(\overline{z})+o_t(1)e^{-t}, \;\; &\mbox{for $N$ and $p$ satisfying \eqref{pN}}\\
                                        & \qquad\mbox{or $p=7$ and $N =4$ with \eqref{1.6},}\vspace{1mm}\\
-f({\overline z})+o_t(1) e^{-\hat{\beta} t}, \;\; & \mbox{for $p \in (\frac{N+2}{6-N},p^*)$ and $N=3$, $4$ or $5$}. \end{array} \right.
\end{array}
\end{equation}
The corresponding characteristic polynomial of \eqref{4.1} is
\begin{eqnarray}
\label{4.1-0}
& & \beta^4+2(N+2\alpha-4) \beta^3+(N^2+6N\alpha+6\alpha^2-10N-24\alpha+20) \beta^2 \nonumber \\
& &\;\;\;\;\;\; +2(N+2\alpha-4)(N \alpha +\alpha^2-N -4\alpha+2) \beta-(p+1) L^{-(p+1)}=0.
\end{eqnarray}
Comparing \eqref{4.1-0} with \eqref{3.5}, it is easy to see that the four roots of \eqref{4.1-0} are given by $\beta_j^{(0)}$ corresponding to $\lambda_0 =0$ for $j=1,2,3,4$,
are given in \eqref{3.6}. Denote
\begin{equation}\label{beta-3} \beta_j=\beta_j^{(0)}\quad \mbox{for $j=1,2,3,4$}.\end{equation}
From the expression of $\beta_j$, we have that, for $N =3$ and $p\in (1,3)$; $N\ge 4$ and $p>1$,
\begin{equation}\label{prop-beta-12} \beta_1, \beta_2 \in \R \quad\mbox{and} \quad \beta_2<2-N-\alpha<-1<0<\beta_1.\end{equation}

As to the roots $\beta_3, \beta_4$, we have:

{\it Claim 5.} When $N=3$ and $p\in (1, 3)$; $N\ge 4$ and $p>1$, the following estimates for $\beta_3$ and $\beta_4$ hold:
$$\begin{array}{lll}
\beta_3,  \beta_4 \in \R, \,\,\beta_4 \le \beta_3,\quad \beta_3\left\{ \begin{array}{lll} \le -1, &\mbox{for}& \left\{\begin{array}{lll}
                                                                                                      p\in (1,p_3^1], &&N=3;\\
                                                                                                      p\in (1,p_c], &&N\in [5,12];\\
                                                                                                      p\in(1,\infty), &&N\ge 13;\end{array}\right.\\
                                                                    \in (-1,0),&\mbox{for}& \hspace{5.5mm}p\in [p_3^2,3),\hspace{7.5mm} N=3; \end{array}\right.
\end{array}$$
$$\begin{array}{lll}
\beta_{3,4}= \ell \pm q i \not\in \R, \quad  \ell \left\{\begin{array}{lll} \le -1, &\mbox{for}&  \left\{\begin{array}{lll}
                                                                                             p\in (p_3^1,\frac{5}{3}], &&N=3;\\
                                                                                             p\in (1,3], &&N=4;\\
                                                                                             p\in (p_c,7],&&N=5;\\
                                                                                             p\in(p_c,\infty), &&N\in [6,12];\end{array}\right. \vspace{1mm}\\
                                                                      \in (-1,0),&\mbox{for}&  \left\{\begin{array}{lll}
                                                                                             p\in (\frac{5}{3},p_3^2), &\,\,&N=3;\\
                                                                                             p\in (3,\infty), &&N=4;\\
                                                                                             p\in (7,\infty),&&N=5.\end{array}\right.
                                                   \end{array}\right.
\end{array}$$
Where $p_c$, $p_3^1, p_3^2$ are given in \eqref{pc} and \eqref{4.00} below.

We now introduce the function
$$ \hbar(p, N):=[4+(N-2)^2]^2-16\rho_0,$$
where $\rho_0=\rho_k|_{k=0}$ is given in \eqref{3.7}.
For $N=3$, solving equation $\hbar (p, 3)=0$, we obtain four foots:
\begin{equation}
\label{4.00}
\begin{array}{lll}\displaystyle p_3^1=\frac{5-\sqrt{13-3\sqrt{17}}}{3+\sqrt{13-3\sqrt{17}}},&\quad& \displaystyle p_3^2=\frac{5+\sqrt{13-3\sqrt{17}}}{3-\sqrt{13-3\sqrt{17}}},\vspace{1mm} \\
 \displaystyle p_3^3=\frac{5+\sqrt{13+3\sqrt{17}}}{3-\sqrt{13+3\sqrt{17}}},&& \displaystyle p_3^4=\frac{5-\sqrt{13+3\sqrt{17}}}{3+\sqrt{13+3\sqrt{17}}}.
\end{array}
\end{equation}
It is easy to check that $p_3^3<p_3^4<1<p_3^1<p_3^2<3$. A simple calculation shows $\hbar(1;3)=9>0$. So, we deduce that
$$ \hbar (p, 3) \left\{\begin{array}{lll} \ge 0,  &\quad& \mbox{for}\quad p\in (1,p_3^1]\cup [p_3^2,3),\vspace{1mm}\\
                                        < 0,   && \mbox{for}\quad p\in (p_3^1, p_3^2).
\end{array}\right.$$
This implies that $\beta_3, \beta_4 \in \R$ for $p\in (1,p_3^1]\cup [p_3^2,3)$ and $\beta_3, \beta_4 \not \in \R$ for
$p \in (p_3^1, p_3^2)$.

For $N=3$ and $p\in (p_3^1, p_3^2)$, we have
\begin{equation}
\label{4.3-1} \Re(\beta_3)=\Re(\beta_4)=\frac{1}{2}-\alpha \left\{\begin{array}{lll}\in (-\frac{3}{2},-1]&\,\,\,& \mbox{for $p\in (p_3^1, \frac{5}{3}]$,}\\
                                                                              \in (-1,-\frac{1}{2}) &&  \mbox{for $p\in (\frac{5}{3}, p_3^2)$.}
                                                                              \end{array}\right.
\end{equation}
For $N=3$ and $p\in (1,p_3^1]\cup [p_3^2,3)$, we see from the representations of $\beta_3$ and $\beta_4$ that $\beta_4 < \beta_3$. Moreover,
\begin{equation}
\label{4.3-2}
\beta_4 \le \beta_3<-1, \quad \mbox{for $N=3$ and $p \in (1,p_3^1]$};
\end{equation}
\begin{equation}
\label{4.3-3}
\beta_4 \le \beta_3 \in (-1,0), \quad \mbox{for $N=3$ and $p \in [p_3^2,3)$}.
\end{equation}
To see \eqref{4.3-2}, we have
\begin{equation}
\label{4.3-4}
\beta_3+1=\frac{1}{2}\Big[3-2\alpha+\sqrt{5-4\sqrt{1+p\alpha(2-\alpha)(\alpha^2-1)}}\Big].
\end{equation}
Note that $3-2\alpha<0$ and $5-(3-2\alpha)^2>0$ for $p\in (1,p_3^1]$. Then
$$[5-(3-2\alpha)^2]^2-16[1+p\alpha(2-\alpha)(\alpha^2-1)]=\frac{128(p-1)(p-3)}{(p+1)^2}<0.$$
This implies $\beta_3+1<0$ and thus \eqref{4.3-2} holds.  To see \eqref{4.3-3}, we notice that $3-2\alpha>0$ for $p\in [p_3^2,3)$. It follows from \eqref{4.3-4} that $\beta_3+1>0$,
i.e. $\beta_3>-1$. We also know that, for $p\in [p_3^2,3)$, $1-2\alpha<0$, $5-(1-2\alpha)^2>0$ and
$$[5-(1-2\alpha)^2]^2-16[1+p\alpha(2-\alpha)(\alpha^2-1)]=\frac{128(p+5)(p-1)(p-3)}{(p+1)^3}<0.$$
These imply that $\beta_3<0$. Therefore, \eqref{4.3-3} holds.

For $N=4$ and $p>1$, we have $\hbar(p, 4)=-\frac{1024p(p+3)(p-1)}{(p+1)^4}<0$. This implies that $\beta_3, \beta_4 \not \in \R$.
At this time,
\begin{equation}
\label{4.3-5}
\Re(\beta_3)=\Re(\beta_4)=-\alpha \left\{\begin{array}{lll}\le -1 &\,\,\,& \mbox{for $p\in (1, 3]$,} \\
                                                                              \in (-1,0) &&  \mbox{for $p\in (3,\infty)$.}\end{array}\right.
\end{equation}

For $5\le N\le 12$ and $p>1$, a direct calculations imply that the equation $\hbar(p, N)=0$ has only one root $p_c$ in $(1, \infty)$ and
\begin{equation}
\label{pc}
p_c=\frac{N+2-\sqrt{4+N^2-4\sqrt{N^2+H_N}}}{6-N+\sqrt{4+N^2-4\sqrt{N^2+H_N}}},\quad \mbox{with}\quad H_N=(N(N-4)/4)^2.
\end{equation}
Moreover,
$$\begin{array}{l} \hbar(1, N)=N^2(N-4)^2>0,\\
\hbar(p, N)|_{p=\infty}=(N-4)(N^2-144)+16(N-20)<0.
\end{array}$$
Hence, when $5\le N\le 12$,
$$\hbar (p, N)\left\{\begin{array}{lll}
>0,&\quad&\mbox{for $p\in(1,p_c)$,}\\
=0, &&\mbox{for $p=p_c$,}\\
<0,&& \mbox{for $p\in (p_c,\infty)$,}\\
\end{array}  \right.$$
which implies that $\beta_3, \beta_4 \in \R$ for $p\in (1,p_c]$; $\beta_3|_{p=p_c}=\beta_4|_{p=p_c}=2-\frac{4}{p_c+1}-\frac{N}{2}<-1$ and
$\beta_3, \beta_4 \not \in \R$ for $p\in (p_c, \infty)$.

When $5\le N\le 12$ and $p\in (p_c, \infty)$, $\beta_{3,4}:=\ell \pm q i$ and it is easy to find that
\begin{equation}
\label{4.3-6}
\ell=\Re(\beta_{3,4})=2-\alpha-\frac{N}{2}
\left\{\begin{array}{lll}\le -1 && \mbox{for $p\in (p_c,\infty)$, $N\in [6,12]$,}\\
 && \,\,\,\,\,\,\,\mbox{$p\in (p_c,7]$, \,\,\,\,$N=5$;} \\
        \in(-1,-\frac{1}{2}) && \mbox{for $p>7$, $N=5$.}
        \end{array}\right.
\end{equation}

When $5\le N\le 12$ and $p\in (1, p_c]$, we have
$$
\beta_{3}+1=\frac{1}{2}\Big[6-2\alpha-N+\sqrt{4+(N-2)^2-4\sqrt{\rho_0}}\Big].
$$
Note that, in this case, $6-2\alpha-N<0$,
$$4+(N-2)^2-(6-2\alpha-N)^2=4+4(2-\alpha)(N-4+\alpha)>0,$$
and
\begin{eqnarray}
\label{4.3-7}
& & [4+(N-2)^2-(6-2\alpha-N)^2]^2-16\rho_0 \nonumber \\
& &\;\;\;\;\;\;\;\;=-\frac{16}{(p+1)^2}(N-1)[(4N-19)p^2+(p^2+2p-3)N+10p-3]<0.
\end{eqnarray}
So, we conclude that $\beta_3 \le \beta_4<-1$ for $5\le N\le 12$ and $p\in (1, p_c]$.

When $N\ge 13$ and $p>1$, a simple calculation shows that the equation $\frac{\partial \hbar}{\partial p}=0$ has no
any solution in $(1, \infty)$, which implies that $\beta_3, \beta_4 \in \R$. From the form of $\beta_3+1$ and \eqref{4.3-7}, we
obtain that $\beta_4 \le \beta_3<-1$.

Our claim 5 follows from the above discussions.

In view of (\ref{4.1}), the ODE theory and arguments similar to those in section 3 imply that, for
$N=4$ and $p>1$; $N\in [5,12]$ and $p\in (p_c,\infty)$,
\begin{eqnarray}
\label{rep-z-1}
{\overline z}(t) &=& M_1 e^{\beta_1 t}+A_2 e^{\beta_2 t} + A_3 e^{\ell t} \cos(q t)+A_4 e^{\ell t} \sin(q t) \nonumber \\
&&\;\;\;\;\;\;-B_1\int_t^{\infty} e^{\beta_1(t-s)} h(s,\overline{z}(s))d s+B_3 \int_T^t e^{\ell (t-s)} \cos[q (t-s)] h(s,\overline{z}(s))d s \nonumber \\
&&\;\;\;\;\;\;+ B_2 \int_T^t e^{\beta_2(t-s)} h(s,\overline{z}(s))d s +B_4 \int_T^t e^{\ell (t-s)} \sin[q (t-s)] h(s,\overline{z}(s))d s;
\end{eqnarray}
for $N =3$ and $p\in (1,p_3^1]\cup [p_3^2,3)$; $N\in [5,12]$ and $p\in (1,p_c]$; $N\ge 13$ and $p>1$,
\begin{eqnarray}
\label{rep-z-2}
{\overline z}(t) &=& M_1 e^{\beta_1 t}+A_2 e^{\beta_2 t}+A_3 e^{\beta_3 t}+A_4 e^{\beta_4 t} \nonumber \\
& & \;\;\;\;\;\;-B_1\int_t^{\infty} e^{\beta_1(t-s)}h(s,\overline{z}(s))d s + B_2\int_T^t e^{\beta_2(t-s)} h(s,\overline{z}(s))d s \nonumber \\
& & \;\;\;\;\;\;+ B_3 \int_T^t e^{\beta_3(t-s)} h(s,\overline{z}(s))d s + B_4 \int_T^t e^{\beta_4(t-s)}h(s,\overline{z}(s))d s,
\end{eqnarray}
where $\ell$ and $q$ are given in Claim 5,
$$h(t,\overline{z}(t))=\left \{ \begin{array}{ll} O({\overline z}^2)+o_t (1) e^{-t},\;\;&\mbox{for $N$ and $p$ satisfying \eqref{pN}}\\
 &\;\;\;\; \mbox{or $p=\frac{N+3}{N-5}$ and $N \geq 6$ with \eqref{1.6},}\\
O({\overline z}^2)+o_t (1) e^{-\hat{\beta} t}, \;\; & \mbox{for $p \in (\frac{N+2}{6-N},p^*)$, $N=3, 5$}. \end{array} \right.
$$
Note that $f({\overline z})=O({\overline z}^2)$. Since ${\overline z}(t) \to 0$ as $t \to \infty$,
we see that, for $N = 3$ and $p\in (p_3^1, p_3^2)$; $N=4$ and $p>1$; $N\in [5,12]$ and $p\in (p_c, \infty)$,
$$|\overline{z}(t)| \leq O(e^{\ell t}) +C \int_t^{\infty} e^{\beta_1 (t-s)} |h(s,\overline{z}(s))| d s + C
\int_T^t e^{\ell (t-s)} |h(s,\overline{z}(s))| d s, $$
and, for $N =3$ and $p \in (1,p_3^1]\cup [p_3^2,3)$; $N\in [5,12]$ and $p\in (1,p_c]$; $N\ge 13$ and $p>1$,
$$ |\overline{z}(t)| \leq O(e^{\beta_3 t}) +C\int_t^{\infty} e^{\beta_1(t-s)} |h(s,\overline{z}(s))| d s
+ C\int_T^t e^{\beta_3(t-s)} |h(s,\overline{z}(s))| d s,$$
where $C>0$ is independent of $T$. Note that we have also used the fact $\beta_2<\ell$. Arguments similar to those in
the proof of Proposition \ref{p3.1} imply that
\begin{equation}
\label{4.1-3}
|\overline{z}(t)|=O(e^{\ell t}), \;\;\; \mbox{i.e.} \;\;\;  |\overline{v}(s)|=O(s^{-\ell}), \;\;\;\;\; \mbox{for}\;\;
                           \left\{  \begin{array}{ll} p\in (p_3^1, p_3^2),& N=3; \\
                                                      p\in (1,\infty),& N=4;\\
                                                      p\in(p_c,\infty), & N\in [5, 12]
                           \end{array}\right.
\end{equation}
and
\begin{eqnarray}
\label{4.1-4}
& & |\overline{z}(t)|=O(e^{\beta_3 t}), \nonumber \\
& & \mbox{i.e.} \;\; |\overline{v}(s)|=O(s^{-\beta_3}), \;\;\;\;\;
\mbox{for}\;\;
\left\{  \begin{array}{ll}p\in (1,p_3^1]\cup [p_3^2,3),& N=3; \\
                                                      p\in (1,p_c],& N\in [5,12];\\
                                                      p\in(1,\infty), & N\in [13, \infty).
\end{array}\right.
\end{eqnarray}
The fact $v(s, \theta)={\overline v} (s)+w (s, \theta)$, Proposition \ref{p3.1}, \eqref{4.1-3}, \eqref{4.1-4} and Claim 5 yield that
\begin{equation}
\label{4.1-5}
|v (s, \theta)|=\left\{\begin{array}{ll}
O(s), \;\; &  \begin{array}{l}\mbox{for}\,\,\, p \in (1,\frac{N+2}{6-N}], N=3, 4, 5;\\
                                       \quad\,\,\,\, p\in (1,\infty), N\ge 6;\end{array} \\
O(s^{\hat{\beta}}), \;\; &\; \mbox{for}\,\,\, p\in (\frac{N+2}{6-N},p^*), N=3,4,5;\\
O(s^{-\ell}), \;\; &\; \mbox{for} \,\,\, p\in (7, \infty), N=4;\\
                   &\; \,\,\,\,\,\,\,\,\,\, p=7,  N=4 \,\,\mbox{with \eqref{1.6}}.
\end{array}\right.
\end{equation}
where $\ell=2-\alpha-\frac{N}{2} \in (-1, 0)$ is given in Claim 5 and $\hat{\beta}=|\beta_3^{(1)}|=5-N-2\alpha \in (0,1)$ for $p\in (\frac{N+2}{6-N},p^*)$ and $N=3, 4, 5$, which is given in \eqref{beta}.
We have also used the facts $-1<\ell<\beta_3^{(1)}<0$ for $p\in (\frac{N+2}{6-N},p^*)$ and $N=3, 4, 5$;
$-1<\beta_3<\beta_3^{(1)}<0$ for $p\in [p_3^2, 3)$ and $N=3$.
Moreover, since
$$|f(v)-f({\overline v})|=p \big|(\xi+L)^{-(p+1)}-L^{-(p+1)}\big | |w|=O(|\xi|)|w|,$$
where $\xi (s, \theta)=\gamma w (s, \theta)+ (1-\gamma) {\overline v}(s)$ with $\gamma \in (0,1)$, the estimate similar to \eqref{4.1-5}
yields that
\begin{equation}
\label{4.1-6}
|\xi (s, \theta)|=\left\{\begin{array}{ll}
O(s), \;\; &  \begin{array}{l}\mbox{for}\,\,\, p \in (1,\frac{N+2}{6-N}], N=3, 4, 5;\\
                                       \quad\,\,\,\, p\in (1,\infty), N\ge 6;\end{array} \\
O(s^{\hat{\beta}}), \;\; &\; \mbox{for}\,\,\, p\in (\frac{N+2}{6-N},p^*), N=3,4,5;\\
O(s^{-\ell}), \;\; &\; \mbox{for} \,\,\, p\in (7, \infty), N=4;\\
                   &\;\,\,\,\,\,\,\,\,\,\, p=7,  N=4 \,\,\mbox{with \eqref{1.6}}.
\end{array}\right.
\end{equation}
Therefore
\begin{equation}
\label{4.1-7}
|f(v)-f({\overline v})|=\left\{\begin{array}{ll}
O(s^2), \;\; &  \begin{array}{l}\mbox{for}\,\,\, p \in (1,\frac{N+2}{6-N}], N=3, 4, 5;\\
                                       \quad\,\,\,\, p\in (1,\infty), N\ge 6;\end{array} \\
O(s^{2\hat{\beta}}), \;\; &\; \mbox{for}\,\,\, p\in (\frac{N+2}{6-N},p^*), N=3,4,5;\\
O(s^{1-\ell}), \;\; &\; \mbox{for} \,\,\, p\in (7, \infty), N=4;\\
                   &\; \,\,\,\,\,\,\,\,\,\, p=7,  N=4 \,\,\mbox{with \eqref{1.6}}.
\end{array}\right.
\end{equation}

Consequently, we have the following lemma.

\begin{lem}
\label{l4.1}
Let $v$ be a solution to \eqref{2.6}. Then there exists
$M=M(v)>0$ such that for $p \in (1, \frac{N+2}{6-N}]$ and $N=3$, $4$, $5$; $p\in (1, \infty)$ and $N\ge 6$,
\begin{equation}
\label{4.1-8}
|\overline{v}(s)| \leq M s, \;\; |\overline{v}' (s)| \leq M, \;\; |\overline{v}''(s)| \leq Ms^{-1}
\end{equation}
and
\begin{equation}
\label{4.1-8-1}
\int_{S^{N-1}} v^2(s,\theta) d \theta \leq M s^2  ;
\end{equation}
For $p\in (\frac{N+2}{6-N}, p^*)$ and $N=3, 4, 5$,
\begin{equation}
\label{4.1-9}
|\overline{v}(s)| \leq M s^{{\hat \beta}}, \;\;\; |\overline{v}'(s)| \leq M s^{{\hat \beta}-1}, \;\;\; |\overline{v}'' (s)| \leq M s^{{\hat \beta}-2}
\end{equation}
and
\begin{equation}
\label{4.1-10}
\int_{S^{N-1}} v^2(s,\theta) d \theta \leq M s^{2\hat{\beta}},\qquad \hat{\beta}=N+2\alpha-5 \in (0,1);
\end{equation}
For $p \in (7, \infty)$ and $N=4$; $p=7$ and $N=4$ with \eqref{1.6},
\begin{equation}
\label{4.1-11}
|\overline{v}(s)| \leq M s^{-\ell}, \;\;\; |\overline{v}'(s)| \leq M s^{-(1+\ell)}, \;\;\; |\overline{v}'' (s)|
\leq M s^{-(2+\ell)},
\end{equation}
and
\begin{equation}
\label{4.1-12}
\int_{S^{N-1}} v^2(s,\theta) d \theta \leq M s^{-2 \ell}.
\end{equation}
where $\ell=2-\alpha-\frac{N}{2}<0$.
\end{lem}

\begin{proof}
Proof of this lemma is similar to that of Lemma 4.1 in \cite{GHZ}. We omit the details here.
\end{proof}

\begin{prop}
\label{p4.2} Suppose that $\kappa \geq 0$ is an integer and $v$ is a solution of \eqref{2.6}. Then there exist $0<s_0<1$
and $M=M(v,\kappa)>0$ $($independent of $s$$)$ such that for $p \in (1, \frac{N+2}{6-N}]$ and $N=3$, $4$, $5$; $p\in (1, \infty)$ and $N\ge 6$,
\begin{equation}
\label{4.1-13}
\max_{|y|=s} |D^\kappa v(y)| \leq M s^{1-\kappa}.
\end{equation}
For $p\in (\frac{N+2}{6-N}, p^*)$ and $N=3, 4, 5$,
\begin{equation}
\label{4.1-14}
\max_{|y|=s} |D^\kappa v(y)| \leq M s^{{\hat \beta}-\kappa}.
\end{equation}
For $p \in (7, \infty)$ and $N=4$; $p=7$ and $N=4$ with \eqref{1.6},
\begin{equation}
\label{4.1-15}
\max_{|y|=s} |D^\kappa v(y)| \leq M s^{-\ell-\kappa}.
\end{equation}
\end{prop}

\begin{proof}
We only show \eqref{4.1-13}. The proofs of \eqref{4.1-14} and \eqref{4.1-15} are similar. We first obtain \eqref{4.1-13} for the case of $\kappa=0$. If we define
$z (t, \theta)=w(s, \theta) =\Sigma_{k=1}^\infty \Sigma_{j=1}^{m_k} w_j^k (s) Q_j^k (\theta)$, we see that
$$\max_{\theta \in S^{N-1}} |z(t, \theta)| \leq \sum_{k=1}^\infty \sum_{j=1}^{m_k} |z_j^k (t)| \max_{\theta \in S^{N-1}} |Q_j^k (\theta)| \leq \sum_{k=1}^\infty \sum_{j=1}^{m_k} D_k |z_j^k (t)|,$$
where $D_k$ is given in \eqref{2.4-1}. Arguments similar to those in the proof of Proposition \ref{p3.1} imply that
there exist $C>0$ independent of $t$ and $T^* \gg 1$ such that, for $t \geq T^*$,
$$\sum_{k=1}^\infty \sum_{j=1}^{m_k} D_k |z_j^k (t)|
=O\Big(\sum_{k=2}^\infty k m_k D_k e^{\beta_4^{(k)} (t-T)} \Big)+O \Big(e^{-(t-T)} \Big) \leq C e^{-t}$$
(note that $\lim_{k \to \infty} \frac{(k+1) m_{k+1} D_{k+1}}{k m_k D_k}=1$) and hence
$$\max_{\theta \in S^{N-1}} |z(t, \theta)| \leq C e^{-t},$$
\begin{equation}
\label{4.1-a}
\max_{\theta \in S^{N-1}} |w(s, \theta)| \leq  C s
\end{equation}
for $0<s<s_0:=e^{-T^*}$.
Therefore, \eqref{4.1-13} with $\tau=0$ can be obtained from \eqref{4.1-a} and the fact that $v(s, \theta)=w(s, \theta)+{\overline v}(s)$.

To obtain \eqref{4.1-13} completely, it is enough to show \eqref{4.1-13} for $\kappa=1$. The other cases are essentially the same by differentiating $w(s, \theta)$. We only need to show
$|\nabla w (y)| \leq C$. Since $|\nabla w|^2=w_s^2+\frac{1}{s^2} |w_\theta|^2$, we need to present the estimates of $w_s^2$ and $|w_\theta|^2$. We see that
$w_s (s, \theta)=\Sigma_{k=1}^\infty \Sigma_{j=1}^{m_k}  (w_j^k)'(s) Q_j^k (\theta)$, then
\begin{equation}
\label{4.1-b}
\max_{\theta \in S^{N-1}} |w_s (s, \theta)| \leq \sum_{k=1}^\infty \sum_{j=1}^{m_k} D_k |(w_j^k)'(s)|.
\end{equation}
For each $\lambda_k=k (N+k-2)$ and $1 \leq j \leq m_k$, we see from the expression of $z_j^k (t)$ in \eqref{3.14-0} and $(w_j^k)' (s)=-(z_j^k)' (t) e^t \; (t=-\ln s)$ that for $0<s<s_0$,
$$D_k |(w_j^k)'(s)| \leq {\tilde M}_k s^{-(\beta_4^{(k)}+1)} \;\;\; \mbox{for $k \geq 2$}$$
and
$$D_1 |(w_j^1)'(s)| \leq {\tilde M}_1.$$
These and \eqref{4.1-b} imply that there is $M_1=M_1(v,s_0)>0$ independent of $s$ such that, for $s \in (0, s_0)$,
\begin{equation}
\label{4.1-c}
\max_{\theta \in S^{N-1}} |w_s (s, \theta)| \leq M_1.
\end{equation}
Note that $\beta_4^{(k)}+1<0$ for $k \geq 2$. Since $|w_\theta (s, \theta)| \leq \Sigma_{k=1}^\infty \Sigma_{j=1}^{m_k} |w_j^k (s)| |(Q_j^k)_\theta|$,
we also obtain that there exists $M_2=M_2 (v, s_0)>0$ independent of $s$ such that for $s \in (0, s_0)$,
\begin{equation}
\label{4.1-d}
\max_{\theta \in S^{N-1}} |w_\theta (s, \theta)| \leq M_2 s.
\end{equation}
(Again by $\lim_{k \to \infty} \frac{(k+1) m_{k+1} E_{k+1}}{k m_k E_k}=1$.)
Therefore, for $s \in (0, s_0)$,
\begin{equation}
\label{4.1-e}
\max_{|y|=s} |\nabla w (y)|^2=\max_{|y|=s} \Big[w_s^2+\frac{1}{s^2} |w_\theta|^2 \Big] \leq {\hat M},
\end{equation}
where ${\hat M}=M_1^2+M_2^2$. Together with $\eqref{4.1-8}_2$, we see
 that \eqref{4.1-13} holds for $\kappa=1$.
This completes the proof of this proposition.
\end{proof}

To study the properties of $v$, we introduce a new function
\begin{equation}
\label{4.1-16}
\tilde{w}(s,\theta)=\frac{w(s,\theta)}{s}.
\end{equation}
It follows from the above arguments that
\begin{equation}
\label{4.1-17}
|{\tilde w} (s, \theta)|=O(1), \;\;\; \mbox{for} \;\;
\left\{ \begin{array}{l}
p \in \left\{
\begin{array}{lll}
(1,\frac{N+2}{6-N}] &&\mbox{when $N=3$ or $5$,}\\
(1,3]\cup (7,\infty) &&\mbox{when $N=4$,}\\
(1,\infty)&&\mbox{when $N\ge 6$},
\end{array}\right. \\
\mbox{$p=7$ and $N=4$ with
\eqref{1.6}}
\end{array}
 \right.
 \end{equation}

For $p \in (\frac{N+2}{6-N}, p^*)$, $N=3$, $4$ or $5$, from \eqref{1.6-1} we have
$$|v(s, \theta)|=o(s^{\hat{\beta}}) \;\; \mbox{for $s$ near 0},$$
where $\hat{\beta}= |\beta_3^{(1)}| =N+2\alpha-5 \in (0,1)$.
There holds in this case
\begin{equation}
\label{4.1-18} |{\overline v}(s)|=o(s^{\hat{\beta}}), \;\;\; |w(s, \theta)|=o(s^{\hat{\beta}}) \;\; \mbox{for $s$ near 0.}
\end{equation}
Arguments similar to those in the proof of Proposition \ref{p3.1} imply that
$$W(s)=O(s^{|\beta_4^{(1)}|}) \;\; \mbox{for $s$ near 0}.$$
Since $\beta_4^{(1)}=-1$ for $p \in (\frac{N+2}{6-N},p^*)$ and $N=3$, $4$, $5$, we see
that $|\tilde{w} (s, \theta)|=O(1)$ also holds when $p \in (\frac{N+2}{6-N}, p^*)$ and $N=3$, $4$, $5$ with \eqref{1.6-1}.

Taking account of equation \eqref{2.11}, we see that ${\tilde w}(s,\theta)$ satisfies the equation:
\begin{equation}
\label{4.2}
\begin{split}& \partial_s^{4} \tilde{w}
+2(9-2 \alpha-N) s^{-1}\partial_s^3 \tilde{w}\\
&\;\;\;\;\;\;\;\;\;\;\;+(N^2+6\alpha N+6\alpha^2-22N-48\alpha+93) s^{-2} \partial_s^2 \tilde{w} \\
&\;\;\;\;\;\;\;\;\;\;\;-(N+2\alpha-7)(2\alpha N+2\alpha^2-14\alpha-5N+21)s^{-3} \partial_s \tilde{w} \\
&\;\;\;\;\;\;\;\;\;\;\;-\big[(N+2\alpha-5)(2N\alpha+2\alpha^2-3N-10\alpha+9)+(p+1)L^{-(p+1)}\big]s^{-4}\tilde{w} \\
&\;\;\;\;\;\;\;\;\;\;\;+2(N\alpha+\alpha^2-6\alpha-2N+9) s^{-4} \Delta_\theta \tilde{w}+2(7-2\alpha-N) s^{-3} \Delta_\theta (\partial_s \tilde{w}) \\
&\;\;\;\;\;\;\;\;\;\;\;+2 s^{-2} \Delta_\theta (\partial_s^2 \tilde{w})+s^{-4} \Delta_\theta^2 \tilde{w}\\
&\;\;\;=s^{-5}[\overline{f(v)}-f(v)]=s^{-4}[{\overline {f'(\xi (s, \theta)) {\tilde w}}}-f'(\xi (s, \theta)) \tilde{w}].
\end{split}
\end{equation}
Now we write
\begin{equation}
\label{4.3}
\tilde{w}(s,\theta)=\sum_{k=1}^{\infty} \sum_{j=1}^{m_k} \tilde{w}_j^k (s)Q_j^k (\theta),
\end{equation}
where $\tilde{w}_j^k (s)=s^{-1} w_j^k (s)$. It is clear that $\overline{\tilde{w}}=0$.
Then, $\tilde{w}_j^k (s)$ satisfies
\begin{equation}
\label{4.4}
\begin{split} & (\tilde{w}_j^k)^{(4)}
+2(9-2 \alpha-N) s^{-1} (\tilde{w}_j^k)^{(3)} \\
&\;\;\;\;\;\;\;\;\;\;\;+(N^2+6\alpha N+6\alpha^2-22N-48\alpha+93-2\lambda_k) s^{-2} (\tilde{w}_j^k)'' \\
&\;\;\;\;\;\;\;\;\;\;\;-(N+2\alpha-7)(2\alpha N+2\alpha^2-14\alpha-5N+21-2\lambda_k)s^{-3} (\tilde{w}_j^k)' \\
&\;\;\;\;\;\;\;\;\;\;\;+\big[(5-N-2\alpha)(2N\alpha+2\alpha^2-3N-10\alpha+9)-(p+1)L^{-(p+1)} \\
&\;\;\;\;\;\;\;\;\;\;\;\;\;\;\;-2(N\alpha+\alpha^2-2N-6\alpha+9)\lambda_k+\lambda_k^2\big]s^{-4}
\tilde{w}_j^k \\
&\;\;\;\;=s^{-4} o_s (1) \tilde{w}_j^k,
\end{split}
\end{equation}
where $\lambda_k=k(N+k-2)$ and
\begin{equation}
\label{4.5}
o_s (1)=\left\{\begin{array}{ll}
O(s), \;\; &  \begin{array}{l}\mbox{for}\,\,\, p \in (1,\frac{N+2}{6-N}], N=3, 4, 5;\\
                                       \quad\,\,\,\, p\in (1,\infty), N\ge 6;\end{array} \\
o(s^{\hat{\beta}}), \;\; &\; \mbox{for}\,\,\, p\in (\frac{N+2}{6-N},p^*), N=3,4,5\,\,\mbox{with \eqref{1.6-1}};\\
O(s^{-\ell}), \;\; &\; \mbox{for} \,\,\, p\in (7, \infty), N=4;\\
                   &\; \,\,\,\,\,\,\,\,\,\, p=7,  N=4 \,\,\mbox{with \eqref{1.6}}.
\end{array}\right.
\end{equation}

Let $t=-\ln s$ and $\tilde{z}_j^k (t)=\tilde{w}_j^k(s)$. Then $\tilde{z}_j^k (t)$
satisfies the following equation:
\begin{equation}
\label{4.6}
\begin{split} & (\tilde{z}_j^k)^{(4)}
+2(N-6+2\alpha) (\tilde{z}_j^k)^{(3)}
+(N^2+6\alpha N+6\alpha^2-16N-36\alpha+50\\
&\;\;\;\;\;\;\;\;\;\;\;\;
-2\lambda_k) (\tilde{z}_j^k)'' +2(N+2\alpha-6)(N\alpha+\alpha^2-6\alpha-2N+7-\lambda_k)(\tilde{z}_j^k)'\\
&\;\;\;\;\;\;\;\;\;\;\;\;+\big[(5-N-2\alpha)(2N\alpha+2\alpha^2-3N-10\alpha+9)-(p+1) L^{-(p+1)} \\
&\;\;\;\;\;\;\;\;\;\;\;\;\;\;\;\;\;\;-2(N\alpha+\alpha^2-2N-6\alpha+9)\lambda_k+\lambda_k^2\big] \tilde{z}_j^k \\
&\;\;\;\;\;=o_t (1) \tilde{z}_j^k,
\end{split}
\end{equation}
where
\begin{equation}
\label{4.7}
o_t (1)=\left\{\begin{array}{ll}
O(e^{-t}), \;\; &  \begin{array}{l}\mbox{for}\,\,\, p \in (1,\frac{N+2}{6-N}], N=3, 4, 5;\\
                                       \quad\,\,\,\, p\in (1,\infty), N\ge 6;\end{array} \\
o(e^{-\hat{\beta}t}), \;\; &\; \mbox{for}\,\,\, p\in (\frac{N+2}{6-N},p^*), N=3,4,5\,\,\mbox{with \eqref{1.6-1}};\\
O(e^{\ell t}), \;\; &\; \mbox{for} \,\,\, p\in (7, \infty), N=4;\\
                   &\; \,\,\,\,\,\,\,\,\,\, p=7,  N=4 \,\,\mbox{with \eqref{1.6}}.
\end{array}\right.
\end{equation}
The corresponding polynomial of \eqref{4.6} is
\begin{equation}
\label{4.8}
\begin{split} &\tilde{\beta}^4
+2(N-6+2\alpha) \tilde{\beta}^3
+(N^2+6\alpha N+6\alpha^2-16N-36\alpha+50-2\lambda_k) \tilde{\beta}^2 \\
&\;\;\;\;\;\;\;\;\;\;+2(N+2\alpha-6)(N\alpha+\alpha^2-6\alpha-2N+7-\lambda_k)\tilde{\beta} \\
&\;\;\;\;\;\;\;\;\;\;+(5-N-2\alpha)(2N\alpha+2\alpha^2-3N-10\alpha+9)-(p+1) L^{-(p+1)} \\
&\;\;\;\;\;\;\;\;\;\;\;\;\;\;-2(N\alpha+\alpha^2-2N-6\alpha+9)\lambda_k+\lambda_k^2=0.
\end{split}
\end{equation}
Solving this equation, we obtain four roots:
$$\tilde{\beta}^{(k)}_j=\beta^{(k)}_j +1, \;\; j=1,2,3,4,$$
where $\beta^{(k)}_j$ is given in \eqref{3.6}. It follows from Claims 1-4 in the proof of Proposition \ref{p3.1} that
$$\tilde{\beta}^{(k)}_2< \tilde{\beta}^{(k)}_4<0<\tilde{\beta}^{(k)}_3<\tilde{\beta}^{(k)}_1 \;\; \mbox{for $k \geq 2$; $N=3$ and $p\in (1, 3)$; $N\ge 4$ and $p>1$}$$
and
$$\begin{array}{ll} \tilde{\beta}^{(1)}_2< \tilde{\beta}^{(1)}_4=0<\tilde{\beta}^{(1)}_3<\tilde{\beta}^{(1)}_1,
\;\;& \mbox{for $N=3, 5$, $p\in (\frac{N+2}{6-N},p^*)$;}\\
    &\,\,\,\,\,\,\,\,\,\mbox{$N =4$, $p \in (3, \infty)$;} \\
\tilde{\beta}^{(1)}_2< \tilde{\beta}^{(1)}_4 \leq \tilde{\beta}^{(1)}_3=0<\tilde{\beta}^{(1)}_1,
\;\; & \mbox{for $N \geq 6$, $p \in (1, \infty)$;}\\
     & \,\,\,\,\,\,\,\,\,\mbox{$N=3, 4, 5$, $p\in (1, \frac{N+2}{6-N}]$}.
\end{array}
$$
From \eqref{3.14-4} and Claim 3, we see that for $k\ge 2$,
\begin{equation}
\label{4.9}
\lim_{s \to 0} \tilde{w}_j^k (s)=0, \;\; \mbox{for $N =3$ and $p\in (1, 3)$; $N\ge 4$  and $p>1$}.
\end{equation}
Moreover, \eqref{4.1-17} and (\ref{4.1-18}) imply that $|{\tilde w}_j^1 (s)| \; (1 \leq j \leq m_1)$ is bounded for $s$ near 0, that is, $|{\tilde z}_j^1 (t)|$ is bounded
for $t$ near $\infty$, provided that $N$ and $p$ satisfy \eqref{pN} or $N =4$ and $p=7$ with \eqref{1.6}; $N =3, 4, 5$ and $p \in (\frac{N+2}{6-N},p^*)$ with \eqref{1.6-1}. It follows from \eqref{4.6} that, for $t$ sufficiently large,
\begin{equation}
\label{4.11}
\begin{split}
 \tilde{z}_j^1 (t) &= C_j^1 +A_{j,2}^1 e^{\tilde{\beta}_2^{(1)}t}+A_{j,4}^1 e^{\tilde{\beta}_4^{(1)}t}\\
 &\;\;\;-B_1^1 \int_t^{\infty}e^{\tilde{\beta}_1^{(1)}(t-s)}O(e^{-s})\tilde{z}_j^1(s)ds
 +B_2^1 \int_T^t e^{\tilde{\beta}_2^{(1)}(t-s)}O(e^{-s})\tilde{z}_j^1(s)ds \\
 &\;\;\;
    +B_3^1 \int_T^{t} O(e^{-s})\tilde{z}_j^1(s)ds + B_4^1 \int_T^t e^{\tilde{\beta}_4^{(1)}(t-s)} O(e^{-s})\tilde{z}_j^1(s)ds
\end{split}
\end{equation}
for $N \geq 6$ and $p\in (1, \infty)$; $N=3, 4, 5$ and $p\in (1, \frac{N+2}{6-N}]$;
\begin{equation}
\label{4.10}
\begin{split}
 \tilde{z}_j^1 (t) =& C_j^1 +A_{j,2}^1 e^{\tilde{\beta}_2^{(1)}t} \\
 &\;-B_1^1 \int_t^{\infty}e^{\tilde{\beta}_1^{(1)}(t-s)}O(e^{\ell s})\tilde{z}_j^1(s)ds
 -B_3^1 \int_t^{\infty}e^{\tilde{\beta}_3^{(1)}(t-s)}O(e^{\ell s})\tilde{z}_j^1(s)ds \\
 &\;+B_2^1 \int_T^t e^{\tilde{\beta}_2^{(1)}(t-s)}O(e^{\ell s})\tilde{z}_j^1(s)ds
    + B_4\int_T^t O(e^{\ell s})\tilde{z}_j^1(s)ds
\end{split}
\end{equation}
for $N=4$ and $p>7$; $N =4$ and $p=7$ with \eqref{1.6};
\begin{equation}
\label{4.12}
\begin{split}
 \tilde{z}_j^1 (t) &= C_j^1 +A_{j,2}^1 e^{\tilde{\beta}_2^{(1)}t}\\
 &\;\;-B_1^1 \int_t^{\infty}e^{\tilde{\beta}_1^{(1)}(t-s)} o(e^{-\hat{\beta} s})\tilde{z}_j^1(s)ds
 -B_3^1 \int_t^\infty e^{\tilde{\beta}_3^{(1)} (t-s)} o(e^{-\hat{\beta} s})\tilde{z}_j^1(s)ds \\
 &\;\;+B_2^1 \int_T^t e^{\tilde{\beta}_2^{(1)}(t-s)} o(e^{-\hat{\beta} s})\tilde{z}_j^1(s)ds
    + B_4^1 \int_T^t  o(e^{-\hat{\beta} s})\tilde{z}_j^1(s)ds
\end{split}
\end{equation}
for $N =3, 4, 5$ and $p \in (\frac{N+2}{6-N},p^*)$ with \eqref{1.6-1}. These imply ${\tilde w}_j^1 (s) \to C_j$ (a constant, maybe 0) as $s \to 0$. Recalling that $Q_1^1 (\theta), \ldots, Q_{m_1}^1 (\theta)$ are the eigenfunctions corresponding to
$\lambda_1=N-1$, we have that
\begin{equation}
\label{4.13}
\tilde{w}(s,\theta)\to V(\theta) \;\;\mbox{as $s\to 0$,}
\end{equation}
for $N$ and $p$ satisfy \eqref{pN}; $N =4$ and $p=7$ with \eqref{1.6}; $N =3, 4, 5$ and $p \in (\frac{N+2}{6-N},p^*)$ with \eqref{1.6-1}.
Here $V(\theta)$ is $0$ or one of the first eigenfunctions of $-\Delta$ on $S^{N-1}$, i.e.
$$\Delta_{\theta} V+(N-1)V=0, \;\;\; \overline{V}=0.$$
Moreover, it is known from Lemma 8.1 of \cite{Zou} that
\begin{equation}
\label{4.24}
V(\theta)=\theta \cdot x_0
\end{equation}
for some $x_0\in \R^N$ fixed and $\theta=\frac{x}{|x|} \in S^{N-1}$.

Combining what have been discussed above with Lemma \ref{l4.1} and Proposition \ref{p4.2},
we have established the following asymptotic expansions near $y=0$ for solutions of \eqref{2.6}.

\begin{thm}
\label{t4.6} Let $v$ be a solution of \eqref{2.6} and $\tilde{w}$ be given by \eqref{4.1-16}. Suppose that $N$ and $p$ satisfy \eqref{pN}; $N =4$ and $p=7$ with \eqref{1.6}; $N =3, 4, 5$ and $p \in (\frac{N+2}{6-N},p^*)$ with \eqref{1.6-1}.
Then $v(y)=\overline{v}(s) + s \tilde{w}(s,\theta)$ where $\overline{v}(s), \; \tilde{w}(s,\theta)$ have the following properties
and

(i) $\overline{v}$ satisfies
$$|\overline{v}(s)|=O(s), \;\; |\overline{v}'(s)|=O(1), \;\; |\overline{v}''(s)|=O(s^{-1}) $$
for $N =3, 4, 5$ and $p \in (1,\frac{N+2}{6-N}]$; $N\ge 6$ and $p\in (1, \infty)$;
$$|\overline{v}(s)|=O(s^{-\ell}),\;\; |\overline{v}'(s)|=O(s^{-(\ell+1)}), \;\;  |\overline{v}''(s)|=O(s^{-(\ell+2)})$$
for $N=4$ and $p\in (7, \infty)$; $N =4$ and $p=7$ with \eqref{1.6};
$$|\overline{v}(s)|=o(s^{\hat{\beta}}), \;\;|\overline{v}'(s)|=o(s^{\hat{\beta}-1}), \;\; |\overline{v}''(s)|=o(s^{\hat{\beta}-2})$$
for $N =3, 4, 5$ and $p \in (\frac{N+2}{6-N},p^*)$ with \eqref{1.6-1}.

(ii) For any nonnegative integers $\kappa$ and $\kappa_1$, there exists a positive constant $M=M(v,\kappa,\kappa_1)$ such that
$$|s^\kappa D^{\kappa_1}_{\theta} D^{\kappa}_s  \tilde{w}(s,\theta)| \leq M, \;\;\; y \in \mathbf{B}_{s_0}:=\{y: \; |y|<s_0\}, \;\; y \neq 0.$$
Moreover, $\tilde{w}$ satisfies
\begin{equation}
\label{4.14}
\lim_{s \to 0} \tilde{w}(s,\theta)=V(\theta) \;\;\; \mbox{uniformly in $C^\kappa (S^{N-1})$,}
\end{equation}
where $V(\theta)$ is $0$ or one of the first eigenfunctions of $-\Delta_\theta$ on $S^{N-1}$.
\end{thm}

Using transformation \eqref{2.5} and arguments similar to those in the proof of Theorem 5.1 of \cite{GHZ}, we obtain immediately from Theorem \ref{t4.6} that the asymptotic expansions for positive entire solutions of \eqref{1.1-0} at $\infty$.

\begin{thm}
\label{t4.7}
Let $N$ and $p$ satisfy \eqref{pN}; $N =4$ and $p=7$ with \eqref{1.6}; $N =3, 4, 5$ and $p \in (\frac{N+2}{6-N},p^*)$ with \eqref{1.6-1}.
Assume that $u$ is a positive entire solution of \eqref{1.1-0} with \eqref{1.4}. Then $(u, -\Delta u)$ admits the expansion:
\begin{equation}
\label{4.15}
\left \{ \begin{split} &u(x)=r^{\alpha}  \Big[L +\varphi(r)+\frac{\psi(r,\theta)}{r} \Big],\\
&w(x):=-\Delta u(x)=-r^{\alpha-2} \Big[L \alpha (N+\alpha-2)+\varphi_1(r)+ \frac{\psi_1(r,\theta)}{r} \Big],
\end{split} \right.
\end{equation}
where
\begin{equation}
\label{4.16}
\left \{\begin{split}& \varphi_1(r)=r^2 \varphi''+(N+2\alpha-1)r \varphi'+\alpha(N+\alpha-2)\varphi , \\
&\psi_1 (r,\theta)=r^2 \psi_{rr}+(N+2\alpha-3) r \psi_r+(\alpha-1)(N+\alpha-3)\psi+r^{-\alpha}\Delta_{\theta} \psi.
\end{split} \right.
\end{equation}
Furthermore, the following properties for $\varphi, \psi, \varphi_1, \psi_1$ are satisfied:

(i) $\varphi (r)=r^{-\alpha}\overline{u} (r)-L$, and there exist $R_0 \;(:=s_0^{-1})$ and a constant $M=M(u)>0$ such that,
for $N =3, 4, 5$ and $p\in (1,\frac{N+2}{6-N}]$; $N\ge 6$ and $p\in (1, \infty)$,
\begin{equation}
\label{4.17}
|\varphi (r)| \leq M r^{-1}, \;\;\; |\varphi' (r)| \leq Mr^{-2}, \;\;\; |\varphi''(r)| \leq Mr^{-3}\;\;\; \mbox{for $r>R_0$},
\end{equation}
\begin{equation}
\label{4.18}
|\varphi_1' (r)| \leq M r^{-1} \;\; \mbox{for $r>R_0$}.
\end{equation}
For $N=4$ and $p\in (7, \infty)$; $N =4$ and $p=7$ with \eqref{1.6},
\begin{equation}
\label{4.19}
|\varphi (r)| \leq Mr^{\ell}, \;\;\; |\varphi'(r)| \leq Mr^{\ell-1}, \;\;\; |\varphi''(r)| \leq Mr^{\ell-2} \;\;\; \mbox{for $r>R_0$},
\end{equation}
\begin{equation}
\label{4.20}
|\varphi_1'(r)| \leq Mr^{\ell} \;\;\; \mbox{for $r>R_0$}.
\end{equation}
For $N =3, 4, 5$ and $p \in (\frac{N+2}{6-N},p^*)$ with \eqref{1.6-1},
\begin{equation}
\label{4.21}
|\varphi (r)|=o(r^{-\hat{\beta}}), \;\;\; |\varphi'(r)|=o(r^{-\hat{\beta}-1}), \;\;\; |\varphi''(r)|=o(r^{-\hat{\beta}-2}) \;\;\; \mbox{for $r>R_0$},
\end{equation}
\begin{equation}
\label{4.22}
|\varphi_1'(r)|=o(r^{-\hat{\beta}}) \;\;\; \mbox{for $r>R_0$}.
\end{equation}

(ii) Let $\kappa$ and $\kappa_1$ be two non-negative integers. Then there exists a positive constant $M=M(u, \kappa, \kappa_1)$
such that
\begin{equation}
\label{4.23}
|r^\kappa D_{\theta}^{\kappa_1} D_r^\kappa \psi(r,\theta)| \leq M, \;\;\; |\psi_1 (r,\theta)| \leq M \;\;\; \mbox{for $r>R_0$}.
\end{equation}

(iii) Let $\kappa$ be a non-negative integer. Then $\psi(r,\theta)$ tends to $V(\theta)$ uniformly in $C^\kappa (S^{N-1})$
as $r \to \infty$, where $V(\theta)$ is given by (\ref{4.24}).
\end{thm}

\vskip1cm
\section{Proofs of Theorems \ref{main}-\ref{main-2}}
\setcounter{equation}{0}

In this section, we present the proofs of Theorems \ref{main}-\ref{main-2} by using the well known moving plane method.

For $\gamma \in \R$, define the hyperplane:
$$\Upsilon_\gamma=\{x=(x_1,x_2,\ldots,x_N) \in \R^N \;|\; x_1=\gamma \}.$$
For any $x \in \R^N$, denote the reflection point of $x$ about $\Upsilon_\gamma$ by $x^{\gamma}$, i.e.
\[x^{\gamma}=(2\gamma-x_1,x_2,\cdots,x_N).\]
We have the following lemma by using Theorem \ref{t4.7}.

\begin{lem}
\label{l5.1} Assume that $N$ and $p$ satisfy \eqref{pN}; $N =4$ and $p=7$ with \eqref{1.6}; $N =3, 4, 5$ and $p \in (\frac{N+2}{6-N},p^*)$ with \eqref{1.6-1}. Let $u$ be a positive entire solution of \eqref{1.1-0}
satisfying \eqref{1.4}, \eqref{1.6} and \eqref{1.6-1} respectively. Then,

(i) if $\gamma^j \in \R \to \gamma$ and $\{x^j\} \to \infty$ with $x_1^j<\gamma^j$, then
\begin{equation}
\label{5.1}
\lim_{j \to \infty} \frac{|x^j|^{2-\alpha}}{\gamma^j-x_1^j} \left[u(x^j)-u((x^j)^\gamma) \right]=-2 \alpha L \gamma-2(x_0)_1,
\end{equation}
where $(x_0)_1$ is the first component of $x_0$ given in \eqref{4.24}.

(ii) Denote $\gamma_0=-\frac{(x_0)_1}{\alpha L}$. Then there exists a constant $M=M(u)>0$ such that
\begin{equation}
\label{5.2}
\frac{\partial u}{\partial x_1} (x) \geq 0,
\end{equation}
if $x_1 \ge \gamma_0+1$ and $|x| \geq M$.
\end{lem}

\begin{proof}
For $N=3$, $4$, $5$ and $p \in (1,\frac{N+2}{6-N}]$; $N\ge 6$ and $p>1$, the proof of this lemma is similar to that of Lemma 6.2 of \cite{GW3}.
For $N =4$ and $p>7$; $N=4$ and $p=7$ with \eqref{1.6}; $N =3, 4, 5$ and $p \in (\frac{N+2}{6-N},p^*)$ with \eqref{1.6-1},
we can obtain the conclusions from the decay rates of $\varphi(r)$, $\varphi_1 (r)$, $\psi(r,\theta)$ and $\psi_1(r,\theta)$ in Theorem \ref{t4.7}.
In fact, we only need to replace the estimate:
 $$\frac{1}{|x^j|^{\alpha-2}(\gamma-x_1^j)}\big[\xi(|x^j|)|x^j|^{\alpha}-\xi(|(x^{j^{\gamma}}|)|x^{j^{\gamma}}|^{\alpha}\big]=O(|x^j|^{-1})\to 0\,\,
 \mbox{as $j\to \infty$}
 $$
in the proof of Lemma 6.2 in \cite{GW3},
by $$\begin{array}{l} \displaystyle\frac{1}{|x^j|^{\alpha-2}(\gamma-x_1^j)}\Big[\varphi(|x^j|)|x^j|^{\alpha}-\varphi(|(x^{j})^{\gamma}|)|(x^{j})^{\gamma}|^{\alpha}\Big] \vspace{1mm}\\
=\left\{ \begin{array}{ll} O(|x^j|^{\ell})\to 0 & \mbox{as $j\to \infty$ for $N =4$, $p>7$; $N=4$, $p=7$ with \eqref{1.6};}\\
                               O(|x^j|^{-\hat{\beta}})\to 0  & \mbox{as $j\to \infty$ for $N =3, 4, 5$,  $p \in (\frac{N+2}{6-N},p^*)$ with \eqref{1.6-1}},
                               \end{array}\right.
                               \end{array}
 $$
here we have used \eqref{4.19} and \eqref{4.21}. This completes the proof of this lemma.
\end{proof}

Assume $w(x)=-\Delta u(x)$ and rewrite \eqref{1.1-0} in the following form:
\begin{equation}
\label{5.3}
\left\{\begin{array}{ll}
-\Delta u=w \;\;\; & \mbox{in $\R^N$},\\
-\Delta w=-u^{-p} \;\;\; & \mbox{in $\R^N$}.
\end{array}
\right.
\end{equation}

Let us recall Lemma 4.2 in \cite{Troy} due to Troy. We obtain readily that

\begin{lem}
\label{l5.2} Let $\gamma \in \R$ and $u$ be a positive entire solution of
\eqref{1.1-0}. Suppose that
$$u(x) \leq u(x^{\gamma}), \;\; u(x) \not \equiv u(x^{\gamma}), \;\; w(x) \leq  w(x^{\gamma}) \;\;\; \mbox{if $x_1<\gamma$}.$$
Then
\begin{equation}
\label{5.4}
u(x)< u(x^\gamma), \;\; w(x)< w(x^\gamma) \;\;\; \mbox{if $x_1<\gamma$}
\end{equation}
and
\begin{equation}
\label{5.5}
\frac{\partial u}{\partial x_1}(x) >0, \;\; \frac{\partial w}{\partial x_1} (x)>0, \;\;\; \mbox{on $\Upsilon_\gamma$},
\end{equation}
where $x^{\gamma}$ is the reflection point of $x$ with respect to $\Upsilon_\gamma$.
\end{lem}

As a consequence of Lemma \ref{l5.2}, we have the following result.

\begin{lem}
\label{l5.3}
Let $\gamma \in \R$, $N$ and $p$ satisfy \eqref{pN}; $N =4$ and $p=7$ with \eqref{1.6}; $N =3, 4, 5$ and $p \in (\frac{N+2}{6-N},p^*)$ with \eqref{1.6-1}. Let $u$ be a positive entire solution of \eqref{1.1-0}
satisfying \eqref{1.4}, \eqref{1.6} and \eqref{1.6-1} respectively. If
$$ u(x)\leq u(x^{\gamma}), \;\; u(x) \not \equiv u(x^{\gamma}) \;\;\; \mbox{for $x_1<\gamma$},$$
then
\begin{equation}
\label{5.6}
u(x)< u(x^{\gamma}),\;\;\;  w(x)< w(x^{\gamma}) \;\; \mbox{for $x_1<\gamma$}.
\end{equation}
\end{lem}

\begin{proof}
Since $u(x) \leq u(x^{\gamma})$, $u(x) \not \equiv u(x^{\gamma})$ for $x_1<\gamma$,
we deduce that
$$\Delta [w(x)-w(x^\gamma)]=u^{-p} (x)-u^{-p}(x^\gamma) \ge 0 \;\;\; \mbox{if $x_1<\gamma$}.$$
It follows from \eqref{4.15}-\eqref{4.23} that
$$ w(x)-w(x^\gamma) \to 0 \;\;\; \mbox{as $|x| \to \infty$}.$$
Moreover, $w(x)=w(x^\gamma)$ on $\Upsilon_\gamma$. The maximum principle yields
$$w(x)-w(x^\gamma) \leq 0 \;\;\; \mbox{if $x_1<\gamma$}.$$
It follows from Lemma \ref{l5.2} that our conclusions in \eqref{5.6} hold.
 \end{proof}

{\bf Proofs of Theorems \ref{main}, \ref{main-1} and \ref{main-2}}

We first show the sufficiency of these theorems. The main idea of the proof is similar to those in \cite{GHZ, Zou}.
We claim that there exists $\gamma'>0$ such that
\begin{equation}
\label{5.7}
u(x)< u(x^{\gamma}), \;\;\;  w(x)< w(x^{\gamma}) \;\;\; \mbox{for $\gamma \geq \gamma'$ and $x_1<\gamma$}.
\end{equation}
Suppose for contradiction that \eqref{5.7} does not hold. Then by Lemma \ref{l5.3}, there exist two sequences $\{\gamma^j\} \to \infty$ and $\{x^j\}$ with $x^j<\gamma^j$ such that
\begin{equation}
\label{5.8}
 u(x^j) \geq u(y^{j}), \;\;\; y^j=(x^j)^{\gamma^j}, \;\;\;  j=1,2, \ldots.
 \end{equation}
Thanks to $y^j$ tends to $\infty$, we see that $u(y^j)$ tends to infinity. In turn $|x^j| \to \infty$.
By Lemma \ref{l5.1},
we must have
$$ x_1^j \leq \gamma_0 +1=-\frac{(x_0)_1}{\alpha L}+1 \;\;\; \mbox{for $j$ large enough}. $$
Thus, it follows that, for any $\gamma_1>\gamma_0+1$,
$$ u(x^j)\geq u(y^j) \geq u((x^j)^{\gamma_1}) \;\;\; \mbox{for $j$ large},$$
since $(x^j)_1^{\gamma^j} \gg (x^j)_1^{\gamma_1}$ for $j$ large and $u(x) \to \infty$ as $|x|\to \infty$. On the other hand,
using Lemma \ref{l5.1} again, we conclude that
$$ 0 \leq \frac{|x^j|^{2+\alpha}}{\gamma_1-x^j_1} \Big[ u(x^j)-u((x^j)^{\gamma_1}) \Big] \to -2 \alpha L \gamma_1-2(x_0)_1<0,$$
since $x_1^j<\gamma_1$. This is a contradiction and \eqref{5.7} follows.

The rest of the proof is same as that of Theorem 1.1 in \cite{GHZ} and \cite{Zou} for the sufficiency
of Theorems \ref{main}, \ref{main-1} and \ref{main-2}. We omit them here.

We now show the necessity of Theorems \ref{main}, \ref{main-1} and \ref{main-2}. Without loss of generality, we assume $x_*=0$. Then, the necessity of Theorem \ref{main} follows from Proposition 8 of \cite{DFG}. To show the necessity of Theorems \ref{main-1} and \ref{main-2}, we first show
a lemma, which describes the behavior of the unique minimal positive radial entire solution of \eqref{1.1-0} at $\infty$.

\begin{lem}
\label{lm5.4}
Assume $N=3$ and $p\in (1,3)$; $N\ge 4$ and $p\in(1, \infty)$. Let $u\in C^4(\R^N)$ be the minimal positive radial entire solution of \eqref{1.1-0}. Then as $r=|x| \to \infty $, there holds:
$$u(r)=Lr^{\alpha}+\left\{
\begin{array}{ll} O(r^{2-\frac{N}{2}}) & \mbox{for $N=3$ and $p\in (p_3^1,p_3^2)$; $N\in [4,12]$ and $p\in(p_c,\infty)$};\\
    O(r^{-1+\alpha})&  \mbox{for $N\in[5, 12]$ and $p\in (1,p_c)$; $N\ge 13$ and $p\in(1,\infty)$}; \\
    O(r^{\beta_3+\alpha})&  \mbox{for $N=3$ and $p\in [p_3^2,3)$},
\end{array}\right.
$$
where $\beta_3=\frac{1}{2}\Big(1-2\alpha+\sqrt{5-4\sqrt{1+p\alpha(2-\alpha)(\alpha^2-1)}}\Big)$ given by \eqref{beta-3}.
\end{lem}

\begin{proof} In radial coordinate $r=|x|$, \eqref{1.1-0} can be written to: for $r\in(0,\infty)$,
 $$u^{(4)}+\frac{2(N-1)}{r}u'''+\frac{(N-1)(N-3)}{r^2}u''-\frac{(N-1)(N-3)}{r^3}u'=-u^{-p}.$$
For the minimal positive radial entire solution $u(r)$ of \eqref{1.1-0}, we know from \cite{DFG} that it satisfies \eqref{condition}.

Inspired by \cite{FGK,GG,GS,GW,JL,Wang}, we introduce the Emden-Fowler transformation
 $$r=e^t,\quad m(t)=e^{-\alpha t}u(e^t)-L,\quad t\in \R.$$
Under this transformation, \eqref{1.1-0} becomes to
\begin{eqnarray}
\label{5-1}
& & m^{(4)}+2(N+2\alpha-4) m'''+(N^2+6N\alpha+6\alpha^2-10N-24\alpha+20) m''  \nonumber \\
& &\;\;\; +2(N+2\alpha-4)(N \alpha +\alpha^2-N -4\alpha+2) m'-(p+1) L^{-(p+1)}m+g(m)=0,
\end{eqnarray}
where $g(m)=(m+L)^{-p}-L^{-p}+pL^{-(p+1)}m$. Note that \eqref{condition} indicates
$$\lim_{t\to \infty}m(t)=0,$$
so for $|t|$ large enough, $g(m)=O(m^2)$. Comparing \eqref{4.1} with \eqref{5-1}, we
find that they have the same characteristic polynomial \eqref{4.1-0} and the eigenvalues $\beta_j$
($j=1,2,3,4$) given in Section 4. Taking account of the properties of $\beta_j$ given in \eqref{prop-beta-12}
and Claim 5, we obtain the presentations of $m(t)$, which are similar to \eqref{rep-z-1} and \eqref{rep-z-2}
in Section 4 except that $h(s,\bar{z}(s))$ is replaced by $g(w)$. The same arguments imply that
 \begin{equation}
 \label{5.8-1}
|m(t)|=\left\{\begin{array}{ll} O(e^{\ell t}) \;
                           &\mbox{for $p\in (p_3^1, p_3^2)$ and $N=3$;
                                                      $p\in (1,\infty)$ and $N=4$};\\
                                                     & \;\;\;\;\;\; \mbox{$p\in(p_c,\infty)$ and $N\in [5, 12]$};\\
O(e^{\beta_3 t})  &\mbox{for $
                           p\in (1,p_3^1]\cup [p_3^2,3)$ and $N=3$; $p\in (1,p_c]$ and $N\in [5,12]$};\\
                                             & \;\;\;\;\;\;  \mbox{$p\in(1,\infty)$ and $N\in [13, \infty)$}.
                           \end{array}\right.
\end{equation}
Note that $\ell=2-\alpha-\frac{N}{2}$ and $\beta_3=\frac{1}{2}\left(1-2\alpha+\sqrt{5-4\sqrt{1+p\alpha(2-\alpha)(\alpha^2-1)}}\right)$ when $N=3$.
We obtain our
desired results by using \eqref{prop-beta-12} and Claim 5 again.
\end{proof}

We continue to show the necessity of Theorems \ref{main-1} and \ref{main-2}.
It follows from Lemma \ref{lm5.4} that if $u$ is the minimal positive
radial entire solution of \eqref{1.1-0}, then, for $r$ sufficiently large, there holds
\begin{equation}
\label{5.9}
r^{-\alpha} u(r)-L=\left\{\begin{array}{lll} O(r^{\ell}), &\;\;&  \mbox{for $p \in (\frac{5}{3},p_3^2)$ and $N=3$; $p\in(3,7]$ and $N=4$};\\
                                                            &\;\;& \;\;\;\;\; \mbox{$p\in (7,\infty)$ and $N=5$}; \\
                                           O(r^{\beta_3}), &\;\;& \mbox{for $p\in [p_3^2,3)$ and $N=3$}.
                        \end{array}\right.
\end{equation}

On the other hand, we can easily check that, for $p\in [p_3^2, 3)$ and $N=3$,
$$\beta_3<-\hat{\beta}=5-3-2\alpha<0.$$
For $(\frac{5}{3},p_3^2)$ and $N=3$; $p\in(3,7)$ and $N=4$; $p\in (7,\infty)$ and $N=5$,
$$\ell=2-\frac{N}{2}-\alpha<5-N-2\alpha=-\hat{\beta}<0$$
and for $p=7$ and $N=4$,
$$\ell=2-\frac{N}{2}-\alpha=-\frac{1}{2}<-\epsilon_0,$$
where $\epsilon_0$ is given in Theorem \ref{main-1}.
It follows from \eqref{5.9} that for $r$ sufficiently large,
$$r^{-\alpha} u(r)-L=\left\{\begin{array}{lll} o(r^{5-N-2\alpha}), &\;\;&  \mbox{for $p \in (\frac{5}{3},3)$ and $N=3$; $p\in(3,7)$ and $N=4$};\\
                                                            &\;\;& \;\;\;\;\; \mbox{$p\in (7,\infty)$ and $N=5$}; \\
                                           o(r^{-\epsilon_0}), &\;\;& \mbox{for $p=7$ and $N=4$}.\end{array}\right.$$
This completes the proof of the necessity of Theorems \ref{main-1} and \ref{main-2} and then completes the proof of Theorems \ref{main}, \ref{main-1} and \ref{main-2}.
\qed

\vskip1cm
\section{Proof of Theorem \ref{main-3}}
\setcounter{equation}{0}

In this section, we present the proof of Theorem \ref{main-3}. To do this, we first obtain the asymptotic behavior
of a non-minimal positive radial entire solution of \eqref{1.1-0}. We know from \cite{DFG} that when $N=3$ and $1<p<3$; $N \geq 4$ and $p>1$, for any fixed $a>0$ and $\infty>b>{\tilde b}$, \eqref{1.2}
admits a unique non-minimal positive radial entire solution $u_{a, b}(r)$ such that
$$
r^{-2} u_{a,b} (r) \in (A_1, A_2) \;\; \mbox{for $r$ sufficiently large},
$$
where $0<A_1<A_2<\infty$.

The following proposition presents the asymptotic behavior of $u_{a,b} (r)$ at $r=\infty$.

\begin{prop}
\label{6-p1.1}
There exists $d>0$ ($d$ depends on $a$ and $b$) such that, for $r$ near $+\infty$,
\begin{equation}
\label{6.0}
\Delta u_{a,b} (r)= \left \{ \begin{array}{ll} d+O(r^{-\min \{N-2, 2(p-1)\}}), \;\; & \mbox{if $p \neq \frac{N}{2}$},\vspace{1mm}\\
d+O(r^{-(N-2)} \ln r), \;\; &\mbox{if $p=\frac{N}{2}$}, \end{array} \right.
\end{equation}
\begin{equation}
\label{6.00}
r^{-2} u_{a,b} (r)= \left \{ \begin{array}{ll} \frac{d}{2N}+O(r^{-\kappa}), \;\; &\mbox{if $p \neq \frac{N}{2}$ and $\min \{N-2, 2(p-1)\} \neq 2$};\vspace{1mm}\\
\frac{d}{2N}+O \left(r^{-\kappa} \ln r \right), \;\; &\mbox{if $p \neq \frac{N}{2}$ and $\min \{N-2, 2(p-1)\}=2$};\vspace{1mm}\\
\frac{d}{2N}+O \left(r^{-2} \right), \;\; &\mbox{if $p=\frac{N}{2}$ and $N \geq 5$};\vspace{1mm}\\
\frac{d}{2N}+O \left(r^{-1} \ln r \right), \;\; &\mbox{if $p=\frac{3}{2}$ and $N=3$};\vspace{1mm}\\
\frac{d}{2N}+O \left(r^{-2} (\ln r)^2 \right), \;\; &\mbox{if $p=2$ and $N=4$},
\end{array} \right.
\end{equation}
where $\kappa=\min \{2, N-2, 2(p-1)\}$.
\end{prop}

\proof We first show
\begin{equation}
\label{6.1}
\Delta u_{a,b} (r) \to d, \;\; r^{-2} u_{a,b} (r) \to \frac{d}{2N} \;\; \mbox{as $r \to \infty$}.
\end{equation}

It is easily seen from the equation in \eqref{1.2} that $\Delta u_{a,b} (r)$ is decreasing in $(0, \infty)$.
Therefore, there are three cases for $\Delta u_{a,b} (r)$:

(i) $\Delta u_{a,b} (r) \to -e<0$ ($e$ may be $+\infty$) as $r \to \infty$,

(ii) $\Delta u_{a,b} (r) \to 0$ as $r \to \infty$,

(iii) $\Delta u_{a,b} (r) \to d>0$ as $r \to \infty$.

We show that the cases (i) and (ii) do not happen. Since
$$
r^{-2} u_{a,b} (r) \in (A_1, A_2) \;\; \mbox{for $r$ sufficiently large},
$$
we have that
\begin{equation}
\label{6.2}
{\underline {\lim}}_{r \to \infty} r^{-2} u_{a,b} (r) \geq A_1>0.
\end{equation}

If (i) occurs, we see that for any small $\epsilon>0$, there is an $R=R(\epsilon)>1$ such that
\begin{equation}
\label{6.3}
\Delta u_{a,b} (r)<-e+\epsilon \;\; \mbox{for $r>R$}.
\end{equation}
(We may assume $0<e<\infty$. If $e=\infty$, we can choose any $0<e_1<\infty$ such that \eqref{6.3} holds.)
This implies
$$r^{N-1} u_{a,b}'(r)-R^{N-1} u_{a,b}'(R) \leq \frac{(-e+\epsilon)}{N} (r^N-R^N),$$
and
$$u_{a,b}'(r) \leq \frac{R^{N-1}}{r^{N-1}} u_{a,b}'(R)+\frac{(-e+\epsilon)}{N} (r-R^N r^{1-N}).$$
Therefore,
\begin{eqnarray*}
u_{a,b} (r) &\leq&  u_{a,b} (R)+\frac{R^{N-1} u_{a,b}'(R)}{2-N} (r^{2-N}-R^{2-N})\\
& & \;\;\;\; +\frac{(-e+\epsilon)}{2 N} (r^2-R^2)+\frac{(-e+\epsilon) R^N}{N (N-2)} (r^{2-N}-R^{2-N}).
\end{eqnarray*}
This implies
$$
{\overline {\lim}}_{r \to \infty} r^{-2} u_{a,b} (r) \leq -\frac{e}{2 N}<0
$$
by sending $\epsilon$ to 0. This contradicts to \eqref{6.2}.

If (ii) occurs, arguments similar to those in the proof of case (i) imply that
$$
{\overline {\lim}}_{r \to \infty} r^{-2} u_{a,b} (r) \leq 0.
$$
This also contradicts to \eqref{6.2}.

Therefore, case (iii) occurs. Clearly using the arguments similar to those in the proof of case (i), we can prove that
$$
\lim_{r \to \infty} r^{-2} u_{a,b} (r)=\frac{d}{2 N},
$$
and then the limits in \eqref{6.1} hold.

To prove the identities in \eqref{6.0}, we define $v(r)=\Delta u(r)-d$. We omit $a,b$ from $u_{a,b}$ in the following. Then $v(r) \to 0$ as $r \to \infty$ and $v(r)$ satisfies the equation
$\Delta v(r)=\Delta^2 u(r)=-u^{-p}$. It follows from \eqref{6.1} that, for $r$ near $+\infty$,
$$\Delta v(r)=O(r^{-2p}).$$
This implies that
$$r^2 v''(r)+(N-1) r v'(r)=O(r^{2 (1-p)}).$$
Making the transformations:
$$w(t)=v(r), \;\;\; t=\ln r,$$
we have that, for $t$ near $\infty$, $w(t)$ satisfies the equation:
$$w''(t)+(N-2) w'(t)=O(e^{2(1-p) t}).$$
The ODE theory implies that for $T \gg 1$ sufficiently large and $t>T$,
\begin{eqnarray*}
w(t)&=& M_1+A_2 e^{(2-N) t}-B_1 \int_t^\infty O (e^{2(1-p) t}) ds\\
& &\;\;\;\;\;\;\;\;\;\;+B_2 \int_T^t e^{(2-N) (t-s)} O(e^{2(1-p) s}) ds.
\end{eqnarray*}
Note that $B_1$ and $B_2$ are independent of $T$. Since $w(t) \to 0$ as $t \to \infty$, we have that $M_1=0$ and we easily see that
$$w(t)=\left \{ \begin{array}{ll} O(e^{-\min \{N-2, \; 2(p-1)\}t}), \;\; &\mbox{if $p \neq \frac{N}{2}$}, \vspace{1mm}\\
O(t e^{-(N-2) t}), \;\; &\mbox{if $p=\frac{N}{2}$}. \end{array} \right.$$
This implies that the identities in \eqref{6.0}
hold. To see the identities in \eqref{6.00}, we define $\varrho(r)=r^{-2} u (r)-\frac{d}{2N}$. Then $\varrho(r) \to 0$ as $r \to \infty$ and $\varrho(r)$ satisfies the equation
$$r^2 \varrho''+(N+3) r\varrho'+2N \varrho=\Delta u (r)-d=\left \{ \begin{array}{ll} O \left(r^{-\min \{N-2, \; 2(p-1)\}} \right), \;\; &\mbox{if $p \neq \frac{N}{2}$},\vspace{1mm}\\
O \left(r^{-(N-2)}\ln r \right), \;\; &\mbox{if $p=\frac{N}{2}$}. \end{array} \right.
$$
Making the transformations:
$$z(t)=\varrho(r), \;\; t=\ln r,$$
we have that, for $t$ near $\infty$, $z(t)$ satisfies the equation
$$z''(t)+(N+2) z'(t)+2N z(t)=\left \{ \begin{array}{ll} O \left(e^{-\min \{N-2, \; 2(p-1)\} t} \right), \;\; &\mbox{if $p \neq \frac{N}{2}$},\vspace{1mm}\\
 O \left(t e^{-(N-2)t}\right), \;\; &\mbox{if $p=\frac{N}{2}$}. \end{array} \right.
$$
Arguments similar to those in the proof of \eqref{6.0} imply that for $t$ near $+\infty$,
$$z(t)=\left \{ \begin{array}{ll} O(e^{-\kappa t}), \;\; &\mbox{if $p \neq \frac{N}{2}$ and $\min \{N-2, 2(p-1)\} \neq 2$},\vspace{1mm}\\
O \left(t e^{-\kappa t}\right), \;\; &\mbox{if $p \neq \frac{N}{2}$ and $\min \{N-2, 2(p-1)\}=2$},\vspace{1mm}\\
O\left(e^{-2 t} \right), \;\; & \mbox{if $p=\frac{N}{2}$ and $N \geq 5$},\vspace{1mm}\\
O \left( t e^{- t} \right), \;\; &\mbox{if $p=\frac{3}{2}$ and $N=3$},\vspace{1mm}\\
O \left(t^2 e^{-2t } \right), \;\; &\mbox{if $p=\frac{4}{2}$ and $N=4$},
\end{array} \right.
$$
where $\kappa=\min \{2, N-2, 2(p-1)\}$. This implies that the identities in \eqref{6.00} hold. Since $u(r)=r^2 \varrho(r)+\frac{d}{2N} r^2$, we have that
$$\Delta (r^2 \varrho(r))=\Delta u (r)-d>0 \;\; \mbox{for $r \in (0, \infty)$.}$$
If we define $\omega (r):=r^2 \varrho(r)$, we see that $\omega'(0)=0$ and hence
$$\omega'(r)>0 \;\; \mbox{for $r \in (0, \infty)$}.$$
The proof of this proposition is completed. \qed

\begin{rem}
We can easily see that for any fixed $a>0$, $d:=d(a, b)>0$ for $b \in ({\tilde b}, \infty)$ is an increasing function of $b$ with
$$\lim_{b \to {\tilde b}^+} d(a, b)=0.$$
We also know that
$$\lim_{b \to \infty} d(a, b)=\infty.$$
\end{rem}

{\bf Proof of Theorem \ref{main-3}.}

Without loss of generality, we assume $x_*=0$ in Theorem \ref{main-3}. The necessity follows from Proposition \ref{6-p1.1}.

To prove the sufficiency of Theorem \ref{main-3}, we need to know more information on the asymptotic behavior of an entire solution $u \in C^4 (\R^N)$ of \eqref{1.1-0} satisfying \eqref{condition-nonminimal}. The main idea is similar to that of the proof of the sufficiency of Theorem \ref{main}.

Let $u \in C^4 (\R^N)$ be an entire solution of \eqref{1.1-0}. We introduce the Kelvin-type transformation:
\begin{equation}
\label{7.1}
v (y)=|x|^{-2} u(x)-D, \;\;\;\; y=\frac{x}{r^2}, \;\; r=|x|>0, \;\; D>0.
\end{equation}
Then $v(y)=v(s, \theta)$ with $s=|y|=r^{-1}$ satisfies $v(s, \theta) \to 0$ as $s \to 0$ and the equation:
\begin{eqnarray}
\label{7.2}
& & \partial_s^4 v-2 (N-3) s^{-1} \partial_s^3v+(N-1)(N-3) s^{-2} \partial_s^2v-(N-1)(N-3) s^{-3} \partial_s v \nonumber \\
& &\;\;\;\;\;\;+2N s^{-4} \Delta_\theta v-2(N-1) s^{-3} \Delta_\theta (\partial_s v)+2 s^{-2} \Delta_\theta (\partial_s^2 v)+s^{-4} \Delta^2_\theta v \nonumber \\
& &\;\;\;\;\;\;+s^{-6+2p}(v+D)^{-p}=0.
\end{eqnarray}

Define
$$w(s, \theta)=v(s, \theta)-{\overline v}(s),$$
where
$${\overline v}(s)=\frac{1}{|S^{N-1}|} \int_{S^{N-1}} v(s, \theta) d \theta.$$
Then ${\overline v}$ and $w$ respectively satisfy
\begin{eqnarray}
\label{7.3}
& &{\overline v}_s^{(4)}-2(N-3)s^{-1}{\overline v}_{sss}+(N-1)(N-3)s^{-2} {\overline v}_{ss}-(N-1)(N-3) s^{-3} {\overline v}_s \nonumber \\
& & \;\;\;\;\;\;+s^{-6+2p}\overline{(v+D)^{-p}}=0
\end{eqnarray}
and
\begin{eqnarray}
\label{7.4}
& & \partial_s^4 w-2(N-3)s^{-1} \partial_s^3w +(N-1)(N-3)s^{-2}\partial_s^2w-(N-1)(N-3) s^{-3} \partial_sw  \nonumber \\
& &\;\;\;\;\;\;+2N s^{-4} \Delta_\theta w-2(N-1) s^{-3} \Delta_\theta (\partial_s w)+2 s^{-2} \Delta_\theta (\partial_s^2w)+s^{-4} \Delta_\theta^2 w \nonumber \\
& &\;\;\;\;\;\;-s^{-4} g(w)=0,
\end{eqnarray}
where
\begin{eqnarray*}
g(w)&=&s^{2p-2}(v+D)^{-p} - s^{2p-2}\overline{(v+D)^{-p}}\\
&=& s^{2(p-1)} \Big[ \Big((v+D)^{-p}-({\overline v}+D)^{-p} \Big)+ {\overline {\Big(({\overline v}+D)^{-p}-(v+D)^{-p} \Big)}} \Big] \\
&=& -ps^{2(p-1)} \Big[(\xi(s,\theta)+D)^{-(p+1)}w(s,\theta)-\overline{(\xi(s,\theta)+D)^{-(p+1)}w(s,\theta)} \Big]
\end{eqnarray*}
and $\xi (s, \theta)$ is between $v(s, \theta)$ and ${\overline v}(s)$. If we define
$$\zeta (s)=\max_{\theta\in S^{N-1}}| -ps^{2(p-1)}(\xi(s,\theta)+D)^{-(p+1)}|
=ps^{2(p-1)}\max_{\theta\in S^{N-1}}|(\xi(s,\theta)+D)^{-(p+1)}| ,$$
we see that
\begin{eqnarray}
\label{addition} \zeta (s)=O(s^{2(p-1)})\;\; \mbox{for $s$ near 0}. \end{eqnarray}
Note that $\xi (s, \theta) \to 0$ as $s \to 0$.

Since ${\overline w}(s)=0$, we have the expansion:
$$w (s, \theta)=\sum_{i=1}^\infty \sum_{j=1}^{m_i} w_j^i (s) Q_j^i (\theta),$$
where $\{Q_1^1 (\theta), Q_2^1 (\theta), \ldots, Q_{m_1}^1 (\theta), Q_1^2 (\theta), \ldots,
Q_{m_2}^2 (\theta), Q_1^3 (\theta), \ldots\}$ is given in Section 2. We also see that $w_j^i (s)$
with $1 \leq j \leq m_i$ satisfies the equation
\begin{eqnarray}
\label{7.5}
& &(w_j^i)^{(4)}-2(N-3)s^{-1} (w_j^i)_{sss}+[(N-1)(N-3)-2 \lambda_i]s^{-2} (w_j^i)_{ss} \nonumber \\
& &\;\;\;-(N-1)[N-3-2 \lambda_k]s^{-3} (w_j^i)_s-(2N\lambda_i-\lambda_i^2) s^{-4} w_j^i \nonumber \\
& &\;\;\;\;\;\;=s^{-4} {\tilde g}_j^i (s),
\end{eqnarray}
where $\lambda_i=i (N+i-2)$, $i=0,1,2, \ldots$ are the eigenvalues of the equation
$-\Delta_{S^{N-1}} Q=\lambda Q$ given by (\ref{2.3})
and
\begin{equation*}
\tilde{g}_j^i (s)=\int_{S^{N-1}} g(w)Q_j^i (\theta) d\theta=-p \int_{S^{N-1}} s^{2(p-1)}(\xi(s,\theta)+D)^{-(p+1)} w(s,\theta)Q_j^i (\theta)d\theta.
\end{equation*}
We see that
$$|\tilde{g}_j^i (s)| \leq C \zeta (s) W(s)=O(s^{2(p-1)}) W(s) \;\; \mbox{for $s$ near 0},$$
where $W(s)=\Big(\int_{S^{N-1}} w^2 (s, \theta) d \theta \Big)^{\frac{1}{2}}$.

Similar to Proposition \ref{3.1}, we have the following result.

\begin{prop}
\label{p7.1}
For $N=3$ and $1<p<3$; $N \geq 4$ and $p>1$, there exist a sufficiently small $0<s_0<\frac{1}{10}$ and $C>0$ independent of $s$ such that for $s \in (0, s_0)$,
\begin{equation}
\label{7.6}
W(s) \leq C s.
\end{equation}
\end{prop}

\proof  Let $t=-\ln s$, $z_j^i (t)=w_j^i (s)$. Then $z_j^i (t)$ satisfies the equation
\begin{eqnarray}
\label{7.7}
& & (z_j^i)^{(4)}+2N (z_j^i)_{ttt}+(N^2+2N-4-2 \lambda_i) (z_j^i)_{tt}+2N (N-2-\lambda_i) (z_j^i)_t \nonumber \\
& &\;\;\;\;\;\;\;-\lambda_i (2N-\lambda_i) z_j^i=f_j^i (t),
\end{eqnarray}
where $f_j^i (t)={\tilde g}_j^i (e^{-t})$. The corresponding polynomial of \eqref{7.7} is
\begin{equation}
\label{7.8}
\nu^4+2N \nu^3+(N^2+2N-4-2 \lambda_i) \nu^2+2N (N-2-\lambda_i) \nu -\lambda_i (2N-\lambda_i)=0.
\end{equation}
Using the Matlab, we obtain four roots of \eqref{7.8}:
\begin{equation}
\label{7.9}
\begin{array}{ll} \nu_1^{(i)}=\frac{1}{2} \Big(2-N+\sqrt{(N-2)^2+4 \lambda_i} \Big),  & \nu_2^{(i)}=\frac{1}{2} \Big(2-N-\sqrt{(N-2)^2+4 \lambda_i} \Big),\\
\nu_3^{(i)}=\frac{1}{2} \left(-2-N+\sqrt{(N-2)^2+4 \lambda_i} \right),  & \nu_4^{(i)}=\frac{1}{2} \left(-2-N-\sqrt{(N-2)^2+4 \lambda_i} \right).
\end{array}
\end{equation}
Therefore, we have
\begin{equation}
\label{7.10}
\nu_1^{(i)}=i, \;\; \nu_2^{(i)}=2-N-i, \;\; \nu_3^{(i)}=i-2, \;\; \nu_4^{(i)}=-N-i.
\end{equation}
We easily see that
$$\nu_4^{(i)}<\nu_2^{(i)}<\nu_3^{(i)}<\nu_1^{(i)}.$$
For $i=1$,
$$\nu_1^{(1)}=1, \;\; \nu_2^{(1)}=1-N, \;\; \nu_3^{(1)}=-1, \;\; \nu_4^{(1)}=-N-1$$
and
$$\nu_4^{(1)}<\nu_2^{(1)}<\nu_3^{(1)}=-1<0<\nu_1^{(1)}.$$
For $i=2$,
$$\nu_1^{(2)}=2, \;\; \nu_2^{(2)}=-N, \;\; \nu_3^{(2)}=0, \;\; \nu_4^{(2)}=-N-2$$
and
$$\nu_4^{(2)}<\nu_2^{(2)}<-1<\nu_3^{(2)}=0<\nu_1^{(2)}.$$
For $i \geq 3$, we see that
$$\nu_4^{(i)}<\nu_2^{(i)}<-1<0<\nu_3^{(i)}<\nu_1^{(i)}.$$

For $i \geq 3$, we see from \eqref{7.7} and ODE theory that for any $T \gg 1$, there are constants $A_{j, k}^i$, $B_{j,k}^i \; (k=1,2,3,4)$ such that, for $t>T$,
$$z_j^i (t)=\sum_{k=1}^4 \Big[ A_{j,k}^i e^{\nu_k^{(i)} t}+B_k^i \int_T^t e^{\nu_k^{(i)} (t-\tau)} f_j^i (\tau) d \tau \Big],$$
where each $A_{j,k}^i$ depends on $T$ and $\nu_k^{(i)}$, but each $B_k^i$ depends only on $\nu_k^{(i)}$. Therefore,
\begin{eqnarray}
\label{7.11}
z_j^i (t) &=& M_{j,1}^i e^{\nu_1^{(i)} t}+M_{j,3}^i e^{\nu_3^{(i)} t}+A_{j, 2}^i e^{\nu_2^{(i)} t}+A_{j,4}^i e^{\nu_4^{(i)} t} \nonumber \\
& &\;\;\;\;\;-B_1^i \int_t^\infty e^{\nu_1^{(i)} (t-\tau)} f_j^i (\tau) d \tau -B_3^i \int_t^\infty e^{\nu_3^{(i)} (t-\tau)} f_j^i (\tau) d \tau \nonumber \\
& &\;\;\;\;\;+B_2^i \int_T^t e^{\nu_2^{(i)} (t-\tau)} f_j^i (\tau) d \tau+B_4^i \int_T^t e^{\nu_4^{(i)} (t-\tau)} f_j^i (\tau) d \tau
\end{eqnarray}
by using that $\int_T^t=\int_T^\infty-\int_t^\infty$. Note that
\begin{equation}
\label{7.12}
\int_t^\infty e^{\nu_1^{(i)} (t-\tau)} f_j^i (\tau) d \tau \to 0, \;\;\; \int_t^\infty e^{\nu_3^{(i)} (t-\tau)} f_j^i (\tau) d \tau \to 0 \;\; \mbox{as $t \to \infty$}.
\end{equation}
Moreover, the facts that $\nu_{4}^{(i)} < \nu_{2}^{(i)}<0$ and $t-\tau>0$ for $\tau\in (T,t)$ imply that
$$ \int^{t}_T e^{\nu_{4}^{(i)}(t-\tau)} |f(\tau)|d\tau  \leq \int^t_T e^{\nu_{2}^{(i)}(t-\tau)} |f(\tau)|d\tau. $$
The facts that $0<\nu_3^{(i)}<\nu_1^{(i)}$ and $t-\tau<0$ for $\tau \in (t, \infty)$ imply that
$$\int_t^\infty e^{\nu_1^{(i)} (t-\tau)} |f_j^i (\tau)| d \tau \leq \int_t^\infty e^{\nu_3^{(i)} (t-\tau)} |f_j^i (\tau)| d \tau.$$
Therefore, since $z_j^i (t) \to 0$ as $t \to \infty$, we see that $M_{j,1}^i=M_{j,3}^i=0$ and there is $C>0$ depending only on $B_{k}^i \;(k=1,2,3,4)$ but independent of $T$ such that
$$|z_j^i(t)| \leq   O( e^{\nu_{2}^{(i)}t}) +C
\int^t_T e^{\nu_{2}^{(i)}(t-\tau)}| f_j^i (\tau)| d\tau
+C\int^\infty_t e^{\nu_{3}^{(i)}(t-\tau)}
|f_j^i (\tau)| d \tau.$$
Arguments similar to those in the proof of \eqref{3.14-1} imply that
\begin{equation}
\label{7.12-1}
|z_j^i (t)|=O(i e^{\nu_2^{(i)} (t-T)})
\end{equation}
for $t>T$ and $i \geq 3$, $1 \leq j \leq m_i$.

For $i=1$, it is known  that $\nu_{4}^{(1)} < \nu_{2}^{(1)}<\nu_{3}^{(1)}=-1<0<\nu_{1}^{(1)}=1$.
The fact that $z_j^1 (t) \to 0$ as $t \to \infty$ implies that $z_j^1(t)$ can be written in the form
\begin{align*}
z_j^1 (t)=& A_{j,2}^1 e^{\nu_{2}^{(1)}t}+A_{j,3}^1 e^{\nu_{3}^{(1)}t}+A_{j,4}^1 e^{\nu_{4}^{(1)}t}\\
&-B_{1}^1 \int^{\infty}\limits_{t}e^{\nu_{1}^{(1)}(t-\tau)} f_j^1 (\tau)d\tau +B_{3}^1 \int^{t}\limits_{T}e^{\nu_{3}^{(1)}(t-\tau)} f_j^1 (\tau)d\tau\\
&+B_{2}^1 \int^{t}\limits_{T}e^{\nu_{2}^{(1)}(t-\tau)} f_j^1 (\tau)d\tau +B_{4}^1 \int^{t}\limits_{T}e^{\nu_{4}^{(1)}(t-\tau)} f_j^1 (\tau)d\tau.
\end{align*}
Arguments similar to those in the proof of \eqref{7.12-1} imply that
\begin{equation}
\label{7.12-2}
|z_j^1 (t)|=O(e^{-(t-T)})
\end{equation}
for $t>T$ and $1 \leq j \leq m_1$.

For $i=2$, it is known  that $\nu_{4}^{(2)} < \nu_{2}^{(2)}<-1<\nu_{3}^{(2)}=0<\nu_{1}^{(2)}$.
The fact that $z_j^2 (t)\to 0$ as $t \to \infty$ implies that $z_j^2 (t)$ can be written in the form
\begin{eqnarray*}
z_j^2 (t)&=&  A_{j,2}^2 e^{\nu_{2}^{(2)}t}+A_{j,4}^2 e^{\nu_{4}^{(2)}t}\\
& &\;\;\;\;-B_{1}^2 \int^\infty_t e^{\nu_{1}^{(2)}(t-\tau)} f_j^2 (\tau)d\tau -B_{3}^2 \int^{\infty}\limits_{t}  f_j^2 (\tau)d\tau\\
& &\;\;\;\;+B_{2}^2 \int^t_T e^{\nu_{2}^{(2)}(t-\tau)} f_j^2 (\tau)d\tau+B_{4}^2 \int^t_T e^{\nu_{4}^{(2)}(t-\tau)} f_j^2 (\tau)d\tau.
\end{eqnarray*}
Similarly, noting (\ref{addition}) we have
\begin{equation}
\label{7.12-3}
|z_j^2 (t)|=O( e^{\nu_2^{(2)} (t-T)})\; (=O(e^{-N(t-T)}))
\end{equation}
for $t>T$ and $1 \leq j \leq m_2$.
Therefore, if we set $Z(t)=W(s)$ with $t=-\ln s$, arguments similar to those in the proof of \eqref{3.18} imply that
\begin{equation}
\label{7.12-4}
Z (t)=O(e^{-t})
\end{equation}
for $t>T_*$ and $T^*=10 T$.

Let $s_0=e^{-T_*}$. We see from \eqref{7.12-4} that there exists $C>0$ such that for $0<s<s_0$,
\begin{equation}
\label{7.12-5}
W (s) \leq C s.
\end{equation}
This completes the proof of this proposition. \qed

\begin{lem}
\label{l7.2}
Let $v$ be a solution of \eqref{7.2}. Then there exist constant $0<s_0<\frac{1}{10}$ and $M=M(v)>0$ such that for $N=3$ and $1<p<3$; $N \geq 4$ and $p>1$; $s \in (0, s_0)$,
\begin{equation}
\label{7.17}
\left \{ \begin{array}{llll} |{\overline v}(s)| \leq M s, & |{\overline v}'(s)| \leq M , & |{\overline v}''(s)| \leq Ms^{-1} \;&\mbox{for $p>\frac{3}{2}$},\vspace{1mm}\\
|{\overline v}(s)| \leq M s^{1-\epsilon}, & |{\overline v}'(s)| \leq M s^{-\epsilon}, & |{\overline v}''(s)| \leq M s^{-1-\epsilon} \;&\mbox{for $p=\frac{3}{2}$},\vspace{1mm}\\
|{\overline v}(s)| \leq M s^{2(p-1)}, & |{\overline v}'(s)| \leq M s^{2(p-1)-1}, & |{\overline v}''(s)| \leq Ms^{2(p-1)-2} \; &\mbox{for $1<p<\frac{3}{2}$,}
\end{array} \right.
\end{equation}
where $0<\epsilon<\frac{1}{100}$ is sufficiently small, and
\begin{equation}
\label{7.18}
\int_{S^{N-1}} v^2 (s, \theta) d \theta \leq\left \{ \begin{array}{ll} M s^2 \;\;&\mbox{for $p>\frac{3}{2}$},\vspace{1mm}\\
   M s^{2(1-\epsilon)} \;\;&\mbox{for $p=\frac{3}{2}$},\vspace{1mm}\\
  M s^{4(p-1)} \;\;&\mbox{for $1<p<\frac{3}{2}$}.
\end{array} \right.
\end{equation}
\end{lem}

\proof We recall that ${\overline v}(s)$
satisfies the equation
\begin{eqnarray*}
& & {\overline v}_s^{(4)}-2(N-3)s^{-1} {\overline v}_{sss}+(N-1)(N-3)s^{-2} {\overline v}_{ss}-(N-1)(N-3)s^{-3} {\overline v}_s \\
& &\;\;\;\;\;- s^{-4} h(\overline{v})= s^{-4}[\overline{h(v)}- h(\overline{v})],
\end{eqnarray*}
where $h(v)=s^{2 (p-1)}(v+D)^{-p}$ and
$$|\overline{h(v)}-h(\bar{v})| \leq \frac{1}{\omega_{N-1}}\int_{S^{N-1}}|h(v)-h(\bar{v})| d\theta \leq  o(W(s))=o(s).$$

Let ${\overline z}(t)={\overline v}(s)$, $t=-\ln s$. Then ${\overline z} (t)$ satisfies the equation
\begin{equation}
\label{7.18-1}
 {\overline z}^{(4)}+2N ({\overline z})_{ttt}+(N^2+2N-4) ({\overline z})_{tt}+2N (N-2) ({\overline z})_t=h(\bar{z})+ o(e^{-t}),
\end{equation}
Note that $h(\bar{z})= s^{2 (p-1)} ( \bar{z}+D)^{-p} =O(e^{-(2 (p-1))t})$ for $t$ near $\infty$. The corresponding polynomial of \eqref{7.18-1} is
\begin{equation}
\label{7.19}
\nu^4+2N \nu^3+(N^2+2N-4) \nu^2+2N (N-2) \nu=0.
\end{equation}
The four roots of \eqref{7.19} are:
\begin{equation}
\label{7.20}
\nu_1^{(0)}=0, \;\; \nu_2^{(0)}=2-N, \;\; \nu_3^{(0)}=-2, \;\; \nu_4^{(0)}=-N.
\end{equation}
The ODE theory implies
\begin{eqnarray}
\label{7.21}
{\overline z}(t) &=& M_1+A_2 e^{-2 t}+A_3 e^{-(N-2) t}+A_4 e^{-Nt} \nonumber \\
& &\;\;\;\;-B_1 \int_t^\infty {\overline f}(\tau) d \tau+B_2 \int_T^t e^{-2 (t-\tau)} {\overline f}(\tau) d \tau \nonumber \\
& &\;\;\;\;+B_3 \int_T^t e^{-(N-2) (t-\tau)} {\overline f}(\tau) d \tau+B_4 \int_T^t e^{-N(t-\tau)} {\overline f}(\tau) d \tau,
\end{eqnarray}
where ${\overline f} (t)=h({\overline z}(t))+o(e^{-t})$.
The fact that ${\overline z}(t) \to 0$ as $t \to \infty$ implies that $M_1=0$. Notice that ${\overline f}(t)=O(e^{-\min \{2(p-1),1\} t})$, we see from \eqref{7.21} that
there exists $T \gg 1$ such that for $t>T$,
\begin{equation}
\label{7.21-1}
|{\overline z}(t)|= \left \{ \begin{array}{ll} O(e^{-t}), \;&\mbox{for $p>\frac{3}{2}$},\vspace{1mm}\\
O(e^{-(1-\epsilon) t}),  \; &\mbox{for $p=\frac{3}{2}$ and sufficiently small $0<\epsilon<\frac{1}{100}$}, \\
O(e^{-2(p-1) t}),  \; &\mbox{for $1<p<\frac{3}{2}$}.
\end{array} \right.
\end{equation}
(Note that when $N=3$ and $p=\frac{3}{2}$, the term $|\int_T^t  e^{-(N-2) (t-\tau)} O(e^{-2(p-1) \tau}) d \tau| \leq O (e^{-t} \ln t)$.)
This implies that $\eqref{7.17}_1$ holds. Differentiating \eqref{7.21} with respect to $t$ once and twice respectively and noticing ${\overline v}'(s)=-{\overline z}'(t) e^t$
and ${\overline v}''(s)=[{\overline z}''(t)+{\overline z}'(t)] e^{2t}$, we easily see that $\eqref{7.17}_2$ and $\eqref{7.17}_3$ hold.
Note that $v(s, \theta)=w(s, \theta)+{\overline v}(s)$, we obtain \eqref{7.18}.
This completes the proof of this lemma. \qed

\begin{lem}
\label{l7.3}
Let $\tau \geq 0$ be an integer and let $v$ be a solution of \eqref{7.2}. Then there exist $0<s_0<\frac{1}{10}$ and $M=M(v, \tau, s_0)>0$ such that for $s \in (0, s_0)$,
\begin{equation}
\label{7.22}
\max_{|y|=s} |D^\tau v(y)| \leq \left \{ \begin{array}{ll} M s^{1-\tau} \;\; &\mbox{for $p> \frac{3}{2}$},\vspace{1mm}\\
 M s^{1-\epsilon-\tau} \;\; &\mbox{for $p=\frac{3}{2}$}, \vspace{1mm}\\
  M s^{2(p-1)-\tau} \;\; &\mbox{for $1<p<\frac{3}{2}$}.
 \end{array} \right.
\end{equation}
\end{lem}

\proof Similar to the proof of Proposition \ref{p4.2}. \qed

Let
$${\tilde w}(s, \theta)=\frac{w(s, \theta)}{s}.$$
Then ${\tilde w}(s, \theta)$ satisfies the equation
\begin{eqnarray}
\label{7.27}
& & \partial_s^4{\tilde w}-2(N-5) s^{-1} \partial_s^3{\tilde w}+(N-3)(N-7) s^{-2} \partial_s^2{\tilde w}+(N-1)(N-3) s^{-3} \partial_s{\tilde w} \nonumber \\
& &\;\;\;\;\;\;-(N-1)(N-3) s^{-4} {\tilde w}+2s^{-4} \Delta_\theta {\tilde w}-2(N-3) s^{-3} \Delta_\theta (\partial_s{\tilde w})  \nonumber \\
& &\;\;\;\;\;\;+2 s^{-2} \Delta_\theta (\partial_s^2{\tilde w})+s^{-4} \Delta_\theta^2 {\tilde w}-s^{-4} g ({\tilde w})=0,
\end{eqnarray}
where
$$ g(\tilde{w})
=-ps^{2(p-1)}[ (\xi(s,\theta)+D)^{-(p+1)} {\tilde w}-\overline{(\xi(s,\theta)+D)^{-(p+1)} {\tilde w}}],$$
where $\xi(s,\theta)$ is between $v(s,\theta)$ and $\bar{v}(s,\theta)$.
We also have
$${\tilde w}(s, \theta)=\sum_{i=1}^\infty \sum_{j=1}^{m_i} {\tilde w}_j^i (s) Q_j^i (\theta), \;\;\;\; {\tilde w}_j^i (s)=\frac{w_j^i (s)}{s}.$$
Then, ${\tilde w}_j^i (s)$ satisfies the equation:
\begin{eqnarray}
\label{7.28}
& &({\tilde w}_j^i)^{(4)}-2(N-5)s^{-1}({\tilde w}_j^i)_{sss}+[(N-3)(N-7)-2 \lambda_i] s^{-2} ({\tilde w}_j^i)_{ss} \nonumber \\
& &\;\;\;\;\;\;+[(N-1)(N-3)+(2N-6) \lambda_i] s^{-3} ({\tilde w}_j^i)_s  \nonumber \\
& &\;\;\;\;\;\;+[-(N-1)(N-3)-2 \lambda_i+\lambda_i^2] s^{-4} {\tilde w}_j^i=s^{-4} {\hat g}_j^i (s),
\end{eqnarray}
where ${\hat g}_j^i (s)=\int_{S^{N-1}} g(\tilde{w})Q_j^i (\theta) d \theta$. We also know that
$$|{\hat g}_j^i (s)| \leq O(s^{2(p-1)}) {\tilde W} (s),$$
where ${\tilde W}(s)=(\int_{S^{N-1}} |{\tilde w} (s, \theta)|^2 d \theta)^{1/2}$.

Let ${\tilde z}_j^i (t)={\tilde w}_j^i (s)$, $t=-\ln s$, ${\tilde Z} (t)={\tilde W} (s)$. We see that ${\tilde z}_j^i (t)$ satisfies the equation (for $t$ near $\infty$):
\begin{eqnarray}
\label{7.29}
& &({\tilde z}_j^i)^{(4)}+2(N-2) ({\tilde z}_j^i)_{ttt}+(N^2-4N+2-2 \lambda_i) ({\tilde z}_j^i)_{tt} \nonumber\\
& &\;\;\;\;-2[N-2+(N-2) \lambda_i] ({\tilde z}_j^i)_t +[-(N-1)(N-3)-2 \lambda_i+\lambda_i^2] {\tilde z}_j^i \nonumber \\
& &\;\;\;\;={\tilde g}_j^i (t),
\end{eqnarray}
where ${\tilde g}_j^i (t)={\hat g}_j^i (s)$ and $|{\tilde g}_j^i (t)| \leq O(e^{-2(p-1) t}) {\tilde Z}(t)$. Since ${\tilde Z}(t)=e^t Z(t)=O(1)$ (see Proposition \ref{p7.1}), we see that
$|{\tilde g}_j^i (t)|=O(e^{-2(p-1) t})$.
The corresponding polynomial of \eqref{7.29} is
\begin{eqnarray}
\label{7.30}
& & \nu^{4}+2(N-2) \nu^3+(N^2-4N+2-2 \lambda_i) \nu^2-2[N-2+(N-2) \lambda_i] \nu  \nonumber \\
& &\;\;\;\;\;\;+[-(N-1)(N-3)-2 \lambda_i+\lambda_i^2]=0,
\end{eqnarray}
which has four roots:
$${\tilde \nu}_k^{(i)}=\nu_k^{(i)}+1, \;\; k=1,2,3,4,$$
i.e.,
\begin{equation}
\label{7.31}
{\tilde \nu}_1^{(i)}=i+1, \;\; {\tilde \nu}_2^{(i)}=3-N-i, \;\; {\tilde \nu}_3^{(i)}=i-1, \;\; {\tilde \nu}_4^{(i)}=1-N-i.
\end{equation}
Since for each $(i,j)$, $|{\tilde z}_j^i (t)|$ is bounded, arguments similar to those in the proof of Proposition \ref{p7.1} imply that $\Sigma_{i=2}^\infty \Sigma_{j=1}^{m_i} |{\tilde z}_j^i (t)| \to 0$ as $t \to \infty$.
We see that, for $i=1$, the four roots are given by
$${\tilde \nu}_1^{(1)}=2, \;\; {\tilde \nu}_2^{(1)}=2-N, \;\; {\tilde \nu}_3^{(1)}=0, \;\; {\tilde \nu}_4^{(1)}=-N.$$
Thus
$${\tilde \nu}_4^{(1)}<{\tilde \nu}_2^{(1)}<-1<0={\tilde \nu}_3^{(1)}<{\tilde \nu}_1^{(1)}$$
and
\begin{eqnarray*}
{\tilde z}_j^1 (t)&=&C+A_{j,2}^1 e^{{\tilde \nu}_2^{(1)} t}+A_{j,4}^1 e^{{\tilde \nu}_4^{(1)} t}\\
& &\;\;\;\;-B_1^1 \int_t^\infty e^{{\tilde \nu}_1^{(1)} (t-\tau)} O(e^{-2(p-1)t}) d \tau-B_3^1 \int_t^\infty O(e^{-2(p-1)t}) d \tau\\
& &\;\;\;\;+B_2^1 \int_T^t e^{{\tilde \nu}_2^{(1)} (t-\tau)} O(e^{-2(p-1)t}) d \tau+B_4^1 \int_T^t e^{{\tilde \nu}_4^{(1)} (t-\tau)} O(e^{-2(p-1)t}) d \tau.
\end{eqnarray*}
This implies that ${\tilde z}_j^1 (t) \to A_j$ ($A_j$ is a constant, maybe 0) as $t \to \infty$. Since $Q_1^1 (\theta), \ldots, Q_{m_1}^1 (\theta)$ are the eigenfunctions corresponding to the eigenvalue $\lambda_1=N-1$, and thus we see that
\begin{equation}
\label{7.32}
\lim_{s \to 0} {\tilde w} (s, \theta)=V(\theta).
\end{equation}

In conclusion, we have the following theorem.

\begin{thm}
\label{t7.4}
Let $v$ be a solution of \eqref{7.2} and ${\tilde w}$ be given in \eqref{7.27}. Then we have

(i) $v(y)={\overline v}(s)+s {\tilde w} (s, \theta)$ satisfies
$$\left \{ \begin{array}{llll} |{\overline v}(s)| \leq M s, & |{\overline v}'(s)| \leq M , & |{\overline v}''(s)| \leq Ms^{-1} \;&\mbox{for $p>\frac{3}{2}$},\\
|{\overline v}(s)| \leq M s^{1-\epsilon}, & |{\overline v}'(s)| \leq M s^{-\epsilon}, & |{\overline v}''(s)| \leq Ms^{-1-\epsilon} \;&\mbox{for $p=\frac{3}{2}$},\\
|{\overline v}(s)| \leq M s^{2(p-1)}, & |{\overline v}'(s)| \leq M s^{2(p-1)-1}, & |{\overline v}''(s)| \leq Ms^{2(p-1)-2} \; &\mbox{for $1<p<\frac{3}{2}$}.
\end{array} \right.$$

(ii) For any non-negative integers $\tau$ and $\tau_1$, there exists $M=M(v, \tau, \tau_1)>0$ such that
\begin{equation}
\label{7.33}
|s^\tau D_\theta^{\tau_1} D_s^\tau {\tilde w} (y)| \leq M, \;\; y \in B_{s_0}, \;\; y \neq 0,
\end{equation}
where $B_{s_0}=\{y \in \R^N: \; |y|<s_0\}$.
Moreover, ${\tilde w}$ satisfies
\begin{equation}
\label{7.34}
\lim_{s \to 0} {\tilde w} (s, \theta)=V(\theta),
\end{equation}
uniformly in $C^\tau (S^{N-1})$, where $V(\theta)$ is given by (\ref{4.24}).
\end{thm}

We obtain from Theorem \ref{t7.4} the asymptotic expansion of $u (x)$ near $|x|=\infty$.

\begin{thm}
\label{t7.5}
Let $N=3$ and $1<p<3$; $N \geq 4$ and $p>1$; $u$ be a solution of \eqref{1.1-0} satisfying \eqref{condition-nonminimal}. Then $u$ admits the expansion:
\begin{equation}
\label{7.35}
u(x)=r^2 \Big[D+\xi (r)+\frac{\eta (r, \theta)}{r} \Big],
\end{equation}
\begin{equation}
\label{7.36}
w(x):=-\Delta u(x)=-2ND+\xi_1 (r)+\frac{\eta_1 (r, \theta)}{r}
\end{equation}
where
$$\xi_1 (r)=-[r^2 \xi''+(N+3) r \xi'+2N \xi],$$
$$\eta_1 (r, \theta)=-[r^2 \eta_{rr}+(N+1) r \eta_{r}+(N-1) \eta+\Delta_\theta \eta].$$
Moreover, the following properties are satisfied:

(i) $\xi (r)=r^{-2} {\overline u}(r)-D$ and there exist $R_0\; (:=s_0^{-1})>0$ and a constant $M=M(u)>0$ such that, for $r>R_0$,
\begin{equation}
\label{7.37}
\left \{ \begin{array}{llll} |\xi (r)| \leq M r^{-1}, & |\xi'(r)| \leq M r^{-2}, & |\xi''(r)| \leq M r^{-3} \;&\mbox{for $p> \frac{3}{2}$},\\
|\xi (r)| \leq M r^{-(1-\epsilon)}, & |\xi'(r)| \leq M r^{-(2-\epsilon)}, & |\xi''(r)| \leq M r^{-(3-\epsilon)} \;&\mbox{for $p=\frac{3}{2}$},\\
|\xi (r)| \leq M r^{-2(p-1)}, & |\xi'(r)| \leq M r^{-2p+1}, & |\xi''(r)| \leq Mr^{-2p} \; &\mbox{for $1<p<\frac{3}{2}$},
\end{array} \right.
\end{equation}
\begin{equation}
\label{7.38}
|\xi_1 (r)| \leq \left \{ \begin{array}{ll} M r^{-1} \; &\mbox{for $p> \frac{3}{2}$},\vspace{1mm}\\
 M r^{-(1-\epsilon)} \; &\mbox{for $p=\frac{3}{2}$},\\
 M r^{-2(p-1)} \; &\mbox{for $1<p<\frac{3}{2}$}.
\end{array} \right.
\end{equation}

(ii) Let $\tau$ and $\tau_1$ be two non-negative integers. Then there exists a positive constant $M:=M(u, \tau, \tau_1)$ such that, for $r>R_0$,
\begin{equation}
\label{7.39}
|r^\tau D_\theta^{\tau_1} D_r^\tau \eta (r, \theta)| \leq M,
\end{equation}
\begin{equation}
\label{7.40}
|\eta_1 (r, \theta)| \leq M.
\end{equation}

(iii) Let $\tau$ be a non-negative integer. Then $\eta (r, \theta)$ tends to $V(\theta)$ uniformly in $C^\tau (S^{N-1})$ as $r \to \infty$,
 where $V(\theta)$ is given by (\ref{4.24}).

\end{thm}

{\bf Completion of the proof of Theorem \ref{main-3}}

We first write \eqref{1.1-0} to a system of equations:
\begin{equation}
\label{7.42}
\left \{ \begin{array}{ll} -\Delta u=v, \;\;\;& \mbox{in $\R^N$},\\
-\Delta v=-u^{-p}, \;\;\; &\mbox{in $\R^N$}. \end{array} \right.
\end{equation}

We now start the procedure of moving-plane.
As a consequence of the expansions of $u(x)$ in Theorem \ref{t7.5}, we have the following lemma.

\begin{lem}
\label{l7.6}
Let $N=3$ and $1<p<3$; $N \geq 4$ and $p>1$; $u$ be a solution of \eqref{1.1-0} satisfying \eqref{condition-nonminimal}. Then,

(i) If $\gamma^j \in \R \to \gamma$ and $\{x^j\} \to \infty$ with $x_1^j<\gamma^j$, then
\begin{equation}
\label{7.43}
\lim_{j \to \infty} \frac{1}{\gamma^j-x_1^j} \Big[u(x^j)-u((x^j)^\gamma) \Big]=-4D \gamma-2 (x_0)_1,
\end{equation}
where $(x_0)_1$ is the first component of $x_0$ given in (\ref{4.24}).

(ii) Define
\begin{equation}
\label{7.44}
\gamma_0=-\frac{(x_0)_1}{2D}.
\end{equation}
Then there exists a constant $M=M(u)>0$ such that
\begin{equation}
\label{7.45}
\frac{\partial u}{\partial x_1} \geq 0 \;\; \mbox{if $x_1 \geq \gamma_0+1$ and $|x| \geq M$}.
\end{equation}
\end{lem}

\proof To prove \eqref{7.43}, without loss of generality, we assume that
$$\lim_{j \to \infty} \frac{x_j}{|x_j|}={\overline \theta} \in S^{N-1}.$$
For simplicity, we also assume that $\gamma^j=\gamma$, $j=1,2, \ldots$ since the following arguments work equally well for the sequence $\{\gamma^j\}$.
Using the the expansion of $u$ in \eqref{7.35}, we have
\begin{eqnarray*}
\frac{1}{\gamma-x_1^j} \Big[u(x^j)-u((x^j)^\gamma) \Big] &=&\frac{1}{\gamma-x_1^j} \Big[D \Big(|x^j|^2-|(x^j)^\gamma|^2 \Big) \Big] \\
& &+\frac{1}{\gamma-x_1^j} \Big[|x^j|^2 \xi (|x^j|)-|(x^j)^\gamma|^2 \xi (|(x^j)^\gamma|) \Big] \\
& &+\frac{1}{\gamma-x_1^j} \Big[|x^j| \eta (|x^j|, \theta^j)-|(x^j)^\gamma| \eta (|(x^j)^\gamma|, (\theta^j)^\gamma) \Big]\\
&=&I+II+III.
\end{eqnarray*}
We have
$$D (|x^j|^2-|(x^j)^\gamma|^2)=-4D \gamma (\gamma-x_1^j)$$
and hence
$$I=-4D \gamma.$$
We also have that there is $\beta_j$ between $|x^j|$ and $|(x^j)^\gamma|$ such that
$$|x^j|^2 \xi (|x^j|)-|(x^j)^\gamma|^2 \xi (|(x^j)^\gamma|)=\Big[2 \beta_j \xi (\beta_j)+\beta_j^2 \xi'(\beta_j) \Big] \frac{-4 \gamma (\gamma-x_1^j)}{|x^j|+|(x^j)^\gamma|},$$
and in turn
\begin{eqnarray*}
II &=& \frac{1}{\gamma-x_1^j} \Big[2 \beta_j \xi (\beta_j)+\beta_j^2 \xi'(\beta_j) \Big] \frac{-4 \gamma (\gamma-x_1^j)}{|x^j|+|(x^j)^\gamma|}\\
&=& \left \{ \begin{array}{ll} O(|x_j|^{-1}) \to 0, \;\; &\mbox{for $p>\frac{3}{2}$}, \\
O(|x_j|^{-(1-\epsilon)}) \to 0, \;\; &\mbox{for $p=\frac{3}{2}$}, \\
O(|x_j|^{-2(p-1)}) \to 0, \;\; &\mbox{for $1<p<\frac{3}{2}$} \end{array} \right.
\end{eqnarray*}
as $j \to \infty$, since $\frac{|(x^j)^\gamma|}{|x^j|} \to 1$ as $j \to \infty$. Here we have used the estimates of $\xi (r)$ and $\xi'(r)$ in \eqref{7.37}. We now write
\begin{eqnarray*}
III &=& \frac{\eta (|(x^j)^\gamma|, (\theta^j)^\gamma)}{\gamma-x_1^j} \Big[|x^j|-|(x^j)^\gamma| \Big] \\
& & \;\;\;+\frac{|x^j|}{\gamma-x_1^j} \Big[\eta (|x^j|, (\theta^j)^\gamma)-\eta (|(x^j)^\gamma|, (\theta^j)^\gamma) \Big] \\
& & \;\;\;+\frac{|x^j|}{\gamma-x_1^j} \Big[ \eta (|x^j|, \theta^j)-\eta (|x^j|, (\theta^j)^\gamma) \Big]\\
&=& III_1+III_2+III_3.
\end{eqnarray*}
As before, by \eqref{7.39} and arguments similar to those in the proof of (8.11) in Lemma 5.2 of \cite{Zou}, we obtain that $III_1=O(|x^j|^{-1}) \to 0$ as $j \to \infty$,
$III_2=O(|x^j|^{-1}) \to 0$ as $j \to \infty$ and $III_3 \to -2(x_0)_1$ as $j \to \infty$. These imply that \eqref{7.43} holds.

To prove \eqref{7.45}, we use \eqref{7.43}. Suppose that \eqref{7.45} is false. Then there exists a sequence $\{x^j\} \to \infty$ such that
$$\frac{\partial u}{\partial x_1} (x^j)<0, \;\; x^j_1 \geq \gamma_0+1, \;\; \forall j \in \mathbb{N}.$$
It follows that there exists a sequence of bounded positive numbers $\{d_j\}$ such that
$$u(x^j)>u(x_{d_j}), \;\; x_{d_j}=x^j+(2 d_j, 0, \ldots, 0), \;\; \forall j \in \mathbb{N}.$$
Let
$$\gamma^j=x_1^j+d_j>x_1^j.$$
We have
\begin{equation}
\label{7.46}
\frac{1}{\gamma^j-x_1^j} \Big[u(x^j)-u((x^j)^\gamma) \Big]>0, \;\; \forall j\in \mathbb{N}.
\end{equation}
There are two possibilities:
$$\lim_{j \to \infty} \inf \gamma^j<\infty, \;\;\; \lim_{j \to \infty} \gamma^j=\infty.$$
If the first case occurs, we choose a convergent subsequence of $\{\gamma^j\}$ (still denoted by $\{\gamma^j\}$) with the limit $\gamma \geq \gamma_0+1$
and apply \eqref{7.43} and \eqref{7.44} to obtain
$$\lim_{j \to \infty} \frac{1}{\gamma^j-x_1^j} \Big[u(x^j)-u((x^j)^\gamma) \Big]=-4D \gamma-2(x_0)_1 \leq -4D<0.$$
This contradicts \eqref{7.46}. We can derive a contradiction for the second case similarly. The proof is a little variant of the proof of Lemma 8.2 of \cite{Zou}. Thus, neither the first nor the
second case can occur and \eqref{7.45} holds. This completes the proof of this lemma. \qed

To complete the proof of the sufficiency, we use moving-plane arguments of the system of equations \eqref{7.42}. The proof is exactly the same as the proof of Theorem 1.1. We omit the details here.
\qed

\begin{rem}
We conjecture that the following conclusion holds:
{\it If $u \in C^4 (\R^N)$ is an entire solution of \eqref{1.1-0} with  $N=3$ and $1<p<3$ or $N\geq 4$ and $p>1$, then $u$ is the minimal radial entire solution of \eqref{1.1-0} about some $x_* \in \R^N$,
if and only if}
\begin{equation}
\label{7.47}
|x|^{-2} u(x) \to 0 \;\; \mbox{as $|x| \to \infty$}.
\end{equation}

This conjecture implies that if $u$ is an entire solution of \eqref{1.1-0} and \eqref{7.47} holds for $u$, then $u$ must have the exact asymptotic behavior at $\infty$:
$$|x|^{-\alpha}u(x)  \to L\;\; \mbox{as $|x| \to \infty$},$$
where $\alpha$ and $L$ are given in \eqref{alpha-L}.
\end{rem}

\end{document}